\documentclass[reqno,a4paper]{amsart}
\usepackage{graphicx}
\usepackage{a4wide}
\usepackage{amssymb} 
\usepackage{dcolumn}
\usepackage{thmtools}
\usepackage{hyperref}
\usepackage{srcltx}

\numberwithin{equation}{section}

\newcommand{\lapl}{\Delta} 
\newcommand{\grad}{\nabla} 
\newcommand{\Real}{\mathbb{R}} 
\newcommand{\Complex}{\mathbb{C}} 
\renewcommand{\natural}{\mathbb{N}} 
\renewcommand{\div}{\mathrm{div}} 
\newcommand{\dd}[1][y]{\if#1y\,\fi{\mathrm d}} 
\newcommand{\compEmb}{\Subset} 
\newcommand{\Jac}[3][]{{P^{(#2,#3)}_{#1}}} 
\newcommand{\wJac}[2]{\chi^{(#1,#2)}} 
\newcommand{\norm}[2][y]{\if#1y\left\fi\lVert#2\if#1y\right\fi\rVert} 
\newcommand{\abs}[2][y]{\if#1y\left\fi\lvert#2\if#1y\right\fi\rvert} 
\newcommand{\psp}[1]{\sp{(#1)}} 
\newcommand{\CC}{\mathrm{C}}
\newcommand{\LL}{\mathrm{L}}
\newcommand{\HH}{\mathrm{H}}
\newcommand{\HZ}{\mathrm{HZ}}
\newcommand{\WZ}{\mathrm{WZ}}
\newcommand{\floor}[1]{\left\lfloor#1\right\rfloor}
\newcommand{\poly}{\Pi}
\newcommand{\poch}[2]{\left(#1\right)_{#2}} 
\newcommand{\rat}{\mathrm{rat}}
\newcommand{\experimentalRate}{\mathrm{egr}}

\newcommand{\nan}{\multicolumn{1}{c}{---}}
\newcommand{\egr}[4]{\multicolumn{1}{c}{$\experimentalRate_{#3}$}}
\newcommand{\ratio}[4]{\multicolumn{1}{c}{$\rat\psp{#1,#2}_{#3,j}$}}

\DeclareMathOperator{\proj}{Proj}

\newcommand\placeholder{\ensuremath{%
  \mathchoice%
   {\mskip\thinmuskip\lower0.2ex\hbox{\scalebox{1.5}{$\cdot$}}\mskip\thinmuskip}}%
   {\mskip\thinmuskip\lower0.2ex\hbox{\scalebox{1.5}{$\cdot$}}\mskip\thinmuskip}%
   {\lower0.3ex\hbox{\scalebox{1.2}{$\cdot$}}}%
   {\lower0.3ex\hbox{\scalebox{1.2}{$\cdot$}}}%
}

\declaretheorem[name=Theorem,numberwithin=section]{thm}
\declaretheorem[name=Lemma,sibling=thm]{lemma}
\declaretheorem[name=Proposition,sibling=thm]{proposition}
\declaretheorem[name=Corollary,sibling=thm]{corollary}
\declaretheorem[name=Remark,sibling=thm,style=remark]{remark}

\newcolumntype{.}{D{.}{.}{3.3}}

\begin{document}
\title[Orthogonal polynomial projection in the unit disk]{Orthogonal polynomial projection error measured in Sobolev norms in the unit disk}
\author{Leonardo E. Figueroa}
\address{CI\textsuperscript{2}MA and Departamento de Ingenier\'ia Matem\'atica\\Universidad de Concepci\'on\\Casilla 160-C\\Concepci\'on, Chile}
\email{lfiguero@ing-mat.udec.cl}
\thanks{The author was supported by the MECESUP project UCO-0713 and the CONICYT project FONDECYT-1130923.}
\date{\today}

\begin{abstract}
We study approximation properties of weighted $\LL^2$-orthogonal projectors onto the space of polynomials of degree less than or equal to $N$ on the unit disk where the weight is of the generalized Gegenbauer form $x \mapsto (1-\abs{x}^2)^\alpha$.
The approximation properties are measured in Sobolev-type norms involving canonical weak derivatives, all measured in the same weighted $\LL^2$ norm.
Our basic tool consists in the analysis of orthogonal expansions with respect to Zernike polynomials.
The sharpness of the main result is proved in some cases and otherwise strongly hinted at by reported numerical tests.

A number of auxiliary results of independent interest are obtained including some properties of the uniformly and non-uniformly weighted Sobolev spaces involved, a Markov-type inequality, connection coefficients between Zernike polynomials and relations between the Fourier--Zernike expansions of a function and its derivatives.
\end{abstract}


\maketitle

\section{Introduction}

The main purpose of this work is proving the analogue on the unit disk of a well
known fact in the case of the interval; namely, in its simplest manifestation,
the orthogonal projector $\proj_N$ mapping $\LL^2(-1,1)$ onto the space of
polynomials of degree less than or equal to $N$, equivalently defined as the
operator returning the truncation at degree $N$ of the Fourier--Legendre series of
its argument, obeys
\begin{equation}\label{LegendreCase}
(\forall\,u\in\HH^l(-1,1)) \quad \norm{ u - \proj_N(u) }_{\HH^1(-1,1)} \leq C \, N^{3/2-l} \norm{u}_{\HH^l(-1,1)},
\end{equation}
where $C > 0$ depends only on $l$ and $\HH^1(-1,1)$ and $\HH^l(-1,1)$ denote standard Sobolev spaces (this was first proved in \cite{CQ:1982}; see \cite[Chapter 5]{CHQZ-I} for detailed proofs of \eqref{LegendreCase}, its analogues for the Chebyshev weight and the periodic unweighted case; see \cite{Guo:2000a} for its analogue for general Gegenbauer weights on the unit interval).
Our main result (\autoref{thm:lossyProjection}) is
\begin{equation}\label{mainResultAnnouncement}
\left(\forall\,u\in\HH^l_w(B^2)\right) \quad \norm{ u - \proj_N(u) }_{\HH^r_w(B^2)} \leq C \, N^{-1/2+2r-l} \norm{u}_{\HH^l_w(B^2)},
\end{equation}
where $B^2$ is the unit disk, $\proj_N$ is the $\LL^2_w(B^2)$-orthogonal projector onto the space of bivariate polynomials of total degree less than or equal to $N$ and $C > 0$ depends only on the integers $1 \leq r \leq l$ and the weight $w$, which in turn is of the generalized Gegenbauer form $x \mapsto (1-\abs{x}^2)^\alpha$, $\alpha > -1$.
The crucial role Fourier--Legendre expansions play in the cited proofs of \eqref{LegendreCase} will be taken up here by Fourier--Zernike expansions; in particular, $\proj_N$ in \eqref{mainResultAnnouncement} can be expressed as the truncation at total degree $N$ of the Fourier--Zernike series of its argument.

The main result, besides being important on its own, has applications in the analysis of polynomial interpolation operators (this is the motivation behind \eqref{LegendreCase} and its analogues in \cite{CQ:1982} and \cite[Chapter 5]{CHQZ-I}) and, because of the relative simplicity of orthogonal expansion truncation operators, has been exploited in the one-dimensional case by the present author to give partial characterizations of approximability spaces involved in the analysis of nonlinear iterative methods for the numerical solution of high-dimensional PDE \cite[Chapter 4]{Figueroa}.

We expect some of the auxiliary results to be useful in the design and the analysis of spectral methods on the unit disk (cf.\ the survey \cite{BY:2011}; see also \cite{VBLOBO:arXiv2015}).
Although we do report on some numerical tests which rely on some of those auxiliary results, our code is only intended to illustrate (when we can prove it) or suggest (otherwise) the sharpness of our main result and not for general use; in particular, no attempt has been made to make it particularly efficient.

We emphasize that neither \eqref{LegendreCase} nor \eqref{mainResultAnnouncement} are best or quasi-best approximation results; in particular, in each case the restriction of the weighted $\LL^2$-orthogonal projector $\proj_N$ does not result in the $\HH^1(-1,1)$- or the $\HH^r_w(B^2)$-orthogonal projector, respectively.
Such weighted Sobolev best approximation results can be found in \cite[Chapter 5]{CHQZ-I} and \cite{Guo:2000a} in the one-dimensional case and in \cite[\S~4]{LiXu:2014} for balls (constant weight). See also \cite[\S~5]{DaiXu:2011} for related results set in a different kind of Sobolev-type space.

\subsection{Structure of this work}

In the rest of this introductory section we briefly provide pointers to relevant literature on the Zernike families of orthogonal polynomials (\autoref{ssc:literature}), introduce some basic notation (\autoref{ssc:notation}) and some results concerning the Jacobi family of univariate polynomials we will use later (\autoref{ssc:Jac}).

In \autoref{sec:PEEUB} we present those auxiliary results that do not depend on the two-dimensional character of our main problem.
One particular result concerning the density of $\CC^\infty(\overline{B^d})$ ($B^d$ being the $d$-dimensional unit ball) functions in a Sobolev-type function space naturally associated with orthogonal polynomials on $B^d$ is of a different character compared to the rest of this work, so we put it in \autoref{sec:density}.

In the main \autoref{sec:2d} we introduce the exact normalization and indexing scheme of Zernike polynomials we will adopt---i.e., that of \cite{Wunsche:2005}---, obtain connection coefficients between Zernike polynomials of different parameters sometimes involving their derivatives and, as a consequence, relations between the expansion coefficients of a function and that of its derivatives (\autoref{ssc:ZPZFS}).
Then, we prove our main result and extend it by complex interpolation (\autoref{ssc:mainResult}) and later prove where we can and otherwise conjecture informed by numerical experiments the sharpness of our main result (\autoref{ssc:sharpness}).

\subsection{Zernike polynomials}\label{ssc:literature}

The families of Zernike or disk polynomials (\cite{BW:1954}, \cite{Koornwinder:1972},
\cite{Koornwinder:1975}, \cite[Chapter 2]{DunklXu:2014}, \cite{Wunsche:2005}) are pairwise $\LL^2_w$-orthogonal in the unit disk, with $w$ of the form $x \mapsto (1-\abs{x}^2)^\alpha$, and play there the role the Gegenbauer or symmetric Jacobi families of polynomials play in the unit interval.
Sometimes (but not in this work) the words ``complex'' or ``generalized'' are prepended if otherwise the names Zernike/disk polynomials are deemed to correspond exclusively to the real-valued or $\alpha = 0$ cases.
These families of polynomials have been used as basis functions for the approximation of functions and the numerical solution of partial differential equations (see the references in \cite[\S 4]{BY:2011} to which we would add \cite{Noll:1976}).
Just like their one dimensional counterparts, the Zernike polynomials are subject to a wealth of useful and sometimes quite elegant identities (cf.\ \cite{Wunsche:2005} mainly; see also \cite{Glaeske:1996}, \cite{SCH:2004}, \cite{JD:2007}, \cite{Waldron:2008}, \cite{Janssen:2014} and \cite{AEHEWG:2015}).
As it is bound to happen with multivariate orthogonal polynomials, the Zernike polynomials are not the only possible family of orthogonal polynomials with respect to the abovementioned weights; cf.\ \cite[\S~2.3]{DunklXu:2014}.

\subsection{Notation}\label{ssc:notation} We denote by $\natural$ the set of
strictly positive integers $\{1, 2, \dotsc \}$ and let $\natural_0 = \{0\} \cup
\natural$. We denote by $\poly^d$ the space of complex polynomials in $d$
variables and by $\poly^d_n$ the subspace of $\poly^d$ consisting of
polynomials of degree at most $n$.
Let $B^d := \{ x \in \Real^d \mid \abs{x} < 1\}$ (i.e., the unit ball of $\Real^d$) and $\mathbb{S}^{d-1} := \partial(B^d) = \{x \in \Real^d \mid \abs{x} = 1\}$ (i.e., unit sphere of $\Real^d$).

We will denote the Lebesgue $d$-dimensional measure of subsets $\Omega \subset
\Real^d$ simply by $\abs{\Omega}$ and integrals of functions $f \colon \Omega
\to \Complex$ with respect to this measure simply by $\int_\Omega f$ or
$\int_\Omega f(x) \dd x$.
We will denote by $\sigma_{d-1}$
the surface measure of $\mathbb{S}^{d-1}$ \cite[Ex.~3.10.82]{Bogachev}. Given
an integrable $f$ over $\Real^d$, its integral can be expressed in generalized
polar form:
\begin{equation}\label{polarIntegral}
\int_{\Real^d} f(x) \dd x = \int_0^\infty \int_{\mathbb{S}^{d-1}} f(r y) \, r^{n-1} \, \sigma_{d-1}(\dd[n] y) \dd r.
\end{equation}

Given an open subset $\Omega$ of an Euclidean space $\Real^d$, a measurable and
almost-everywhere non-negative and finite \emph{weight} function $w\colon
\Omega \to \Real$ and $m \in \natural_0$ let
\begin{equation}\label{L2}
\LL^2_w(\Omega) := \left\{ u \colon \Omega \to \Complex \text{ Lebesgue measurable} \mid \norm{u}_{\LL^2_w(\Omega)} := \left( \int_\Omega \abs{u}^2\, w \right)^{1/2} < \infty \right\},
\end{equation}
\begin{subequations}\label{Hm}
\begin{equation}
\HH^m_w(\Omega) := \left\{ u \in \LL^2_w(\Omega) \mid \norm{u}_{\HH^m_w(\Omega)} := \left( \sum_{k=0}^m \abs{u}_{\HH^k_w(\Omega)}^2 \right)^{1/2} < \infty \right\},
\end{equation}
where in turn the seminorms $\abs{\placeholder}_{\HH^k_w(\Omega)}$ are defined by
\begin{equation}
\abs{u}_{\HH^k_w(\Omega)} := \Big( {\sum_{\abs{\alpha} = k} \norm{\partial_\alpha u}_{\LL^2_w(\Omega)}^2} \Big)^{1/2}.
\end{equation}
\end{subequations}
The $\LL^2_w(\Omega)$ are Hilbert spaces and under the additional condition
$w^{-1} \in \LL^1_{\mathrm{loc}}(\Omega)$ so are the $\HH^m_w(\Omega)$
(cf.\ \cite{KO}). All the weight  functions used in this work satisfy these
conditions.

Given $a \in \Complex$ and $n \in \natural_0$ the Pochhammer symbol
$\poch{a}{n}$ is defined as $\prod_{k=0}^{n-1} (a+k)$. Due to the empty product
convention $\poch{a}{0} = 1$ for any $a \in \Complex$. Also, $\poch{a}{m+n} =
\poch{a}{m} \poch{a+m}{n}$. If $a \notin -\natural_0$, $\poch{a}{n} =
\Gamma(a+n)/\Gamma(a)$, where $\Gamma$ is the gamma function (cf.\
\cite[\S~1]{AAR:1999}), which in turn is finite and non-zero on $\Complex \setminus
(-\natural_0)$ and obeys $\Gamma(z+1) = z \, \Gamma(z)$ with $z$ over the same set. Besides
these properties we will also use the asymptotic formula (cf.\
\cite[\S~4.5]{Olver:1997})
\begin{equation}\label{GRA}
(\forall\,(a, b) \in \Complex \times \Complex) \quad \frac{\Gamma(z+a)}{\Gamma(z+b)} \sim z^{a-b} \quad \text{as } \Re(z) \to +\infty;
\end{equation}
i.e., the limit of the ratio of both sides is $1$.

We denote the forward difference operator with respect to some index $j$ by
$\Delta_j$; that is, $\Delta_j(f_j) = f_{j+1} - f_j$.
We will denote compact inclusion and compact embedding relations with the symbol $\Subset$.
Lastly, we will denote generic positive constants by $C$ with or without sub- and superscripts, tildes, hats, etc.\ and they may vary from line to line and even from expression to expression.

\subsection{Jacobi polynomials}\label{ssc:Jac}

Let $\alpha, \beta > -1$ and let
$\wJac{\alpha}{\beta}\colon (-1,1) \rightarrow \Real$ be the function defined
by $\wJac{\alpha}{\beta}(t) = (1-t)^\alpha (1+t)^\beta$.
The Jacobi polynomial of parameter $(\alpha, \beta)$ and degree $n$,
denoted by $\Jac[n]{\alpha}{\beta}$ is defined as the member of said degree of
the orthogonalization of the sequence of monomials $(x \mapsto x^n)_{n \in
\natural_0}$ with respect to the $\LL^2_{\wJac{\alpha}{\beta}}(-1,1)$ inner
product together with the normalization condition $\Jac[n]{\alpha}{\beta}(1) =
\binom{n+\alpha}{n}$ (cf.\ \cite[\S~4.1]{Szego:1975}). $\Jac[n]{\alpha}{\beta}$
is also a polynomial with respect to $\alpha$ and $\beta$
\cite[\P~4.22.1]{Szego:1975}.

In \cite[Theorem~7.1.3]{AAR:1999} we find the \emph{connection coefficients}
which allow for expressing $\Jac[n]{\gamma}{\beta}$ in terms of the
$\Jac[k]{\alpha}{\beta}$, $k \in \{0, \dotsc, n\}$; namely,
\begin{multline}\label{Jac-from-AAR}
\Jac[n]{\gamma}{\beta}
= \frac{\poch{\beta+1}{n}}{\poch{\alpha+\beta+2}{n}}\\
\times \sum_{k=0}^n \frac{ \poch{\gamma-\alpha}{n-k} \poch{\alpha+\beta+1}{k} (\alpha+\beta+2k+1) \poch{\beta + \gamma + n + 1}{k} }{ \Gamma(n-k+1) \poch{\beta+1}{k} (\alpha+\beta+1) \poch{\alpha+\beta+n+2}{k} } \Jac[k]{\alpha}{\beta}.
\end{multline}
We note that on account of the continuity of the Jacobi polynomials with
respect to their parameters this relation is still valid if $\alpha + \beta =
-1$ if the above coefficients are replaced by their corresponding limits.

\section{Polynomial eigenfunctions on the Euclidean unit ball}\label{sec:PEEUB}

\subsection{Polynomial eigenfunctions their associated Fourier series}

Let $d \in \natural$
and let $\rho \colon B^d \rightarrow \Real$ be the function defined by
\begin{equation*}
(\forall \, x \in B^d) \quad \rho(x) := 1-\abs{x}^2.
\end{equation*}
Following \cite[eq.~(3.1.2)]{DunklXu:2014} for every $\alpha > -1$ we define
the space of polynomials orthogonal with respect to
$\LL^2_{\rho^\alpha}(B^d)$ of degree exactly $N$ as
\begin{equation*}
\mathcal{V}_N(\LL^2_{\rho^\alpha}(B^d)) = \left\{ P \in \poly^d_N \colon \forall\,Q \in \poly^d_{N-1},\ \langle P, Q \rangle_{\LL^2_{\rho^\alpha}(B^d)} = 0 \right\}.
\end{equation*}
It transpires from the theory exposed in Section 3.2 of \cite{DunklXu:2014} (using the fact that $\alpha > -1$) that
$\bigcup_{k=0}^N \mathcal{V}_k(\LL^2_{\rho^\alpha}(B^d))$ spans $\poly^d_N$
and, consequently,
$\bigcup_{k \geq 0} \mathcal{V}_k(\LL^2_{\rho^\alpha}(B^d))$ spans $\poly^d$. Also, in
\cite[Theorem 8.1.3]{DunklXu:2014}, it was proved that members of
$\mathcal{V}_N(\LL^2_{\rho^\alpha}(B^d))$ have to satisfy any of the equivalent
equations below (we commit the usual abuse of notation consisting of only
selectively omitting the independent variable and its components):
\begin{subequations}\label{twoForms}
\begin{gather}
\label{DunklXuForm1}
\lapl P - \sum_{j=1}^d \partial_j \left(x_j \left[ 2 \alpha \, P + \sum_{i=1}^d x_i \partial_i P \right]\right) = -(N+d)(N + 2\alpha) P,\\
\label{DaiXuForm2}
\sum_{i=1}^d \rho^{-\alpha} \partial_i \left( \rho^{\alpha+1} \partial_i P \right) + \sum_{1 \leq i < j \leq d} (x_j \partial_i - x_i \partial_j)^2 P = -N(N+d+2\alpha) P.
\end{gather}
\end{subequations}
The form \eqref{DunklXuForm1} is from equation~5.2.3 of \cite{DunklXu:2014} and
the form \eqref{DaiXuForm2} is the combination of equation~5.9 and Proposition
7.1 of \cite{DaiXu:2010}; their equivalence can be checked directly.

Let $\mathcal{H}^d_n$ denote the linear space of harmonic polynomials
homogeneous of degree $n$ on $\Real^d$. The spherical harmonics of degree $n$ are the restrictions of
members of $\mathcal{H}^d_n$ to the unit sphere $\mathbb{S}^{d-1}$. Spherical harmonics of differing degrees are $\LL^2(\sigma_{d-1})$-orthogonal. Let $(Y^d_{n,\nu})_{\nu =
1}^{\dim(\mathcal{H}^d_n)}$ be a $\LL^2(\sigma_{d-1})$-orthogonal basis of $\mathcal{H}^d_n$. Then, given
$N \in \natural_0$, the functions $P^N_{j,\nu}(\rho^\alpha; \placeholder)
\colon B^d \rightarrow \Real$ defined by
\begin{equation}\label{BOVN}
(\forall\,x \in B^d) \quad P^N_{j,\nu}(\rho^\alpha;x) := \Jac[j]{\alpha}{N-2j + \frac{d-2}{2}}(2\abs{x}^2-1) \, Y^d_{N-2j,\nu}(x),
\end{equation}
where $\nu \in \{1,\dotsc,\dim(\mathcal{H}^d_{N-2j})\}$, $j \in
\{0,\dotsc,\floor{N/2}\}$, are polynomials and form a
$\LL^2_{\rho^\alpha}(B^d)$-orthogonal basis of
$\mathcal{V}_N(\LL^2_{\rho^\alpha}(B^d))$ (cf.\ \cite[Proposition 5.2.1]{DunklXu:2014}) and, as members of the latter,
satisfy the equations in \eqref{twoForms}. The above considerations motivate the introduction of the
index sets
\begin{subequations}\label{indices}
\begin{equation}
\label{degreeNIndices}
(\forall\,N \in \natural_0) \quad \mathcal{I}^d_N := \left\{ (j, \nu) \mid j \in \{0, \dotsc, \floor{N/2}\},\ \nu \in \{1, \dotsc, \dim(\mathcal{H}^d_{N-2j})\} \right\}
\end{equation}
and
\begin{equation}
\label{allIndices}
\mathcal{I}^d := \left\{ (N, j, \nu) \mid N \in \natural_0,\ (j,\nu) \in \mathcal{I}^d_N \right\}.
\end{equation}
\end{subequations}
We also introduce a notation for the squared norms of the polynomials comprising the above bases.
\begin{equation}\label{squaredNorm}
(\forall \, (N, j, \nu) \in \mathcal{I}^d) \quad h^N_{j,\nu}(\rho^\alpha)
:= \norm{P^N_{j,\nu}(\rho^\alpha; \placeholder)}_{\LL^2_{\rho^\alpha}(B^d)}^2
\end{equation}

From \eqref{DaiXuForm2} it is apparent that solutions of the equations in
\eqref{twoForms} are eigenfunctions of the eigenvalue problem
\begin{equation}\label{strong-EV}
L\psp{\alpha}(u)
:= -\frac{1}{\rho^\alpha} \div\left( \rho^{\alpha+1} \grad u \right) - \sum_{1\leq i<j\leq d} (x_j\partial_i - x_i\partial_j)^2 u = \lambda u.
\end{equation}
In particular, each polynomial $P^N_{j,\nu}(\rho^\alpha;\placeholder)$
defined in \eqref{BOVN} is an eigenfunction of \eqref{strong-EV} with
associated eigenvalue
\begin{equation}\label{eigenvalues}
\lambda_{\alpha,N} = N(N+d+2\alpha),
\end{equation}
which depends only on the dimension $d$, the total degree $N$ and the
singularity parameter $\alpha$ and not on the indices $j$ and $\nu$
of the particular spherical harmonic involved.

For the purposes of approximation in $\rho^\alpha$-weighted Sobolev spaces,
we introduce the following variational eigenvalue problem
reformulation of \eqref{strong-EV}: Find $(\lambda,u) \in \Complex \times
\left( \HZ_\alpha(B^d) \setminus \{0\} \right)$ such that
\begin{equation}\label{wrong-EV-problem}
\int_{B^d} \grad u \cdot \overline{\grad v} \, \rho^{\alpha+1} + \sum_{1 \leq i < j \leq d} \int_{B^d} (x_j \partial_i u - x_i \partial_j u) \overline{(x_j \partial_i v - x_i \partial_j v)} \, \rho^\alpha\\
= \lambda \int_{B^d} u \, \overline{v} \, \rho^\alpha
\end{equation}
for all $v \in \HZ_\alpha(B^d)$; here,
\begin{equation}\label{definition-of-HZ}
\HZ_\alpha(B^d) = \overline{\CC^\infty(\overline{B^d})}^{\norm{\placeholder}_{\WZ_\alpha(B^d)}}
\end{equation}
where the Hilbert space $\WZ_\alpha(B^d)$ is, in turn, defined by
\begin{subequations}\label{definition-of-WZ}
\begin{gather}
\WZ_\alpha(B^d) := \left\{ v \in \LL^2_{\rho^\alpha}(B^d) \colon \norm{v}_{\WZ_\alpha(B^d)} < \infty\right\},\\
\norm{v}_{\WZ_\alpha(B^d)} := \bigg( \norm{v}_{\LL^2_{\rho^\alpha}(B^d)}^2 + \norm{\grad v}_{[\LL^2_{\rho^{\alpha+1}}(B^d)]^d}^2 + \sum_{1\leq i<j\leq d} \norm{x_j \partial_i v - x_i \partial_j v}_{\LL^2_{\rho^\alpha}(B^d)}^2 \bigg)^{1/2}.
\end{gather}
\end{subequations}

\begin{remark}\hfill
\begin{enumerate}
\item The differential operator $x_j
\partial_i - x_i \partial_j$ can be interpreted as an angular derivative (cf.\
\cite[\S~2.1]{DaiXu:2010}). Thus, roughly speaking, the space $\WZ_\alpha(B^d)$
is the subspace of $\LL^2_{\rho^\alpha}(B^d)$ whose members have their angular
derivatives in $\LL^2_{\rho^\alpha}(B^d)$ and their radial derivative in the
larger space $\LL^2_{\rho^{\alpha+1}}(B^d)$.
\item The arguments put forth in \cite{KO} are readily adapted to the presence
of the non-standard differential operator $x_j \partial_i - x_i \partial_j$ in
order to guarantee that $\WZ_\alpha(B^d)$ is indeed a Hilbert space.
\item If $\alpha \geq 0$ then $\HZ_\alpha(B^d) = \WZ_\alpha(B^d)$; that is,
$\CC^\infty(\overline{B^d})$ functions are dense in $\WZ_\alpha(B^d)$. This is proved in
\autoref{cor:WZ-density-nonNeg} in \autoref{sec:density}.
\end{enumerate}
\end{remark}

We start the study of \eqref{wrong-EV-problem} with a number of basic results
on its interaction with the $\HZ_\alpha(B^d)$ and
$\HH^k_{\rho^\alpha}(B^d)$ spaces. While doing this we will often make
silent use of the fact that
\begin{equation}\label{rho-dist-equivalent}
(\forall\, x \in B^d) \quad \operatorname{dist}(x,\partial B^d) \leq \rho(x) \leq 2\, \operatorname{dist}(x,\partial B^d),
\end{equation}
for many results in the cited literature are stated in terms of spaces weighted
with such distance-to-the-boundary functions.

\begin{proposition}\label{pro:basic} Let $\alpha > -1$. Then,
\begin{enumerate}
\item\label{it:compEmb} $\HZ_\alpha(B^d) \compEmb \LL^2_{\rho^\alpha}(B^d)$.
\item\label{it:density} The space of polynomials defined over $B^d$ is dense in
$\LL^2_{\rho^\alpha}(B^d)$.
\item\label{it:density2} $\HZ_\alpha(B^d)$ is dense in
$\LL^2_{\rho^\alpha}(B^d)$.
\end{enumerate}
\end{proposition}
\begin{proof} Setting $\Omega = B^d$, $\kappa = 1$, $p = q = 2$, $\beta =
\alpha+1$ and $\alpha = \alpha$ in Theorem 8.8 of \cite{OG:1989} we find that
$\HH^1_{\rho^{\alpha+1}}(B^d) \compEmb \LL^2_{\rho^\alpha}(B^d)$. The
observation that $\HZ_\alpha(B^d) \subseteq \WZ_\alpha(B^d)$ is continuously
embedded in $\HH^1_{\rho^{\alpha+1}}(B^d)$ completes the proof of part
\ref{it:compEmb}.

As $\int_{B^d} \exp(\abs{y}) \rho(y)^\alpha \dd y < \infty$, the hypotheses of
\cite[Theorem~3.2.18]{DunklXu:2014} are satisfied and so its thesis, namely
\ref{it:density}, is obtained. Part \ref{it:density2} is a direct corollary of
part \ref{it:density}.
\end{proof}

We prove that the
polynomials in \eqref{BOVN} are eigenfunctions of \eqref{wrong-EV-problem} as well. Then,
we will appeal to the Hilbert--Schmidt theory to exploit this fact; a
terse formulation of the former lies below followed by some consequences upon the behavior of
the generalized Fourier series with respect to said polynomials.

\begin{lemma}\label{lem:Hilbert-Schmidt} If $\alpha > -1$, then the eigenvalue
problem \eqref{wrong-EV-problem} has a complete system of solutions
(eigenpairs) with a countably infinite set of finite-multiplicity and isolated
eigenvalues which diverge to $+\infty$ and whose associated eigenfunctions
allow for orthogonal expansions of both $\LL^2_{\rho^\alpha}(B^d)$ and
$\HZ_\alpha(B^d)$. These expansions are subject to Parseval's identity.
\end{lemma}
\begin{proof}
Because of \autoref{pro:basic} this stems from the spectral theory of compact
self-adjointed operators in Hilbert spaces (see, for example, \cite[Theorem
VI.15]{ReedSimon} and \cite[Section 4.2]{Zeidler}).
\end{proof}

\begin{lemma}\label{lem:allArePolys} Let $\alpha > -1$.
\begin{enumerate}
\item\label{it:eigenpairs} The pairs $(\lambda_{\alpha,N},
P^N_{j,\nu}(\rho^\alpha; \placeholder))$ indexed by $(N, j, \nu) \in \mathcal{I}^d$
(cf.\ \eqref{BOVN} and \eqref{eigenvalues}) form a complete system of
eigenpairs of \eqref{wrong-EV-problem}.
\item\label{it:eigenpairs-consequences} Given $u \in \LL^2_{\rho^\alpha}(B^d)$,
on defining
\begin{equation}\label{coefficient}
(\forall \, (N,j,\nu) \in \mathcal{I}^d) \quad \hat u\psp{\alpha}_{N,j,\nu} := \left. \left\langle u, P^N_{j,\nu}(\rho^\alpha; \placeholder) \right\rangle_{\LL^2_{\rho^\alpha}(B^d)} \right/ h^N_{j,\nu}(\rho^\alpha),
\end{equation}
the series
\begin{subequations}\label{EPC}
\begin{equation}\label{expansion}
\left(\forall\,u\in\LL^2_{\rho^\alpha}(B^d)\right) \quad u = \sum_{(N,j,\nu) \in \mathcal{I}^d} \hat u\psp{\alpha}_{N,j,\nu} \, P^N_{j,\nu}(\rho^\alpha; \placeholder)
\end{equation}
converges in $\LL^2_{\rho^\alpha}(B^d)$ in general and in $\WZ_\alpha(B^d)$ if
$u \in \HZ_\alpha(B^d)$ and there hold the Parseval identities
\begin{equation}\label{Parseval-0}
\left(\forall\,u\in\LL^2_{\rho^\alpha}(B^d)\right) \quad \norm{u}_{\LL^2_{\rho^\alpha}(B^d)}^2
= \sum_{(N,j,\nu) \in \mathcal{I}^d} \abs{\hat u\psp{\alpha}_{N,j,\nu}}^2 \, h^N_{j,\nu}(\rho^\alpha)
\end{equation}
and
\begin{equation}\label{Parseval-1}
\left(\forall\,u\in\HZ_\alpha(B^d)\right) \quad \norm{u}_{\WZ_\alpha(B^d)}^2
= \sum_{(N,j,\nu) \in \mathcal{I}^d} (1 + \lambda_{\alpha, N}) \, \abs{\hat u\psp{\alpha}_{N,j,\nu}}^2 \, h^N_{j,\nu}(\rho^\alpha).
\end{equation}
\end{subequations}
\end{enumerate}
\end{lemma}
\begin{proof} Let us abbreviate $\lambda = \lambda_{\alpha, N}$ and $P =
P^N_{j,\nu}(\rho^\alpha;\placeholder)$ for the moment. Then, if $v \in
\CC^\infty(\overline{B^d})$,
\begin{multline}\label{weak-to-strong}
\lambda \langle P, v \rangle_{\LL^2_{\rho^\alpha}(B^d)}
= \langle L\psp{\alpha}(P), v \rangle_{\LL^2_{\rho^\alpha}(B^d)}\\
= -\int_{B^d} \div\left(\rho^{\alpha+1} \grad P\right) \overline{v} - \sum\limits_{1\leq i<j\leq d} \int_{B^d} \left((x_j \partial_i - x_i \partial_j)^2 P\right) \overline{v} \, \rho^\alpha\\
= \int_{B^d} \rho^{\alpha+1} \grad P \cdot \overline{\grad v} + \sum\limits_{1\leq i<j\leq d} \int_{B^d} (x_j \partial_i P - x_i \partial_j P) \overline{(x_j \partial_i v - x_i \partial_j v)} \rho^\alpha,
\end{multline}
where the first equality comes from the fact that $(\lambda, P)$ is an
eigenpair of \eqref{strong-EV}, the second is an immediate consequence of the
definition of $L\psp{\alpha}$ in the same equation and the third comes from
applying the divergence theorem and using the facts that both $\rho^{\alpha+1}$
and $x_j \hat\nu_i - x_i \hat\nu_j$ (here $\hat\nu$ is the unit outward vector defined on
$\partial B^d$) vanish at $\partial B^d$ and that $(x_j \partial_i - x_i
\partial_j) \rho^\alpha \equiv 0$. Now, as per definition
\eqref{definition-of-HZ} $\CC^\infty(\overline{B^d})$ is dense in
$\HZ_\alpha(B^d)$, the identity \eqref{weak-to-strong} is also valid for all
$v$ in the latter space. Therefore, $P$ is an eigenvalue of
\eqref{wrong-EV-problem} with eigenvalue $\lambda$. Incidentally, this
justifies calling \eqref{wrong-EV-problem} the weak form of \eqref{strong-EV}.

The fact that the polynomials $P^N_{j,\nu}(\rho^\alpha;\placeholder)$ form a
\emph{complete} system of eigenfunctions of \eqref{wrong-EV-problem} is then a
consequence of part \ref{it:density} of \autoref{pro:basic} and the fact that
the $P^N_{j,\nu}(\rho^\alpha,\placeholder)$, for fixed $N$, span each space
$\mathcal{V}_N(\LL^2_{\rho^\alpha}(B^d))$, which, in turn, collectively span
the space of all polynomials defined on $B^d$; hence part \ref{it:eigenpairs}.
Having identified a complete system of eigenpairs we can put the orthogonal
expansions and Parseval's identity alluded to in \autoref{lem:Hilbert-Schmidt}
in the form given in \eqref{EPC}, thus giving part
\ref{it:eigenpairs-consequences}.
\end{proof}

\begin{proposition}\label{pro:chainIngredient} Let $\alpha > -1$. Then,
\begin{enumerate}
\item\label{it:map} For every $k \in
\natural_0$, $L\psp{\alpha}$ is a continuous map between
$\HH^{k+2}_{\rho^\alpha}(B^d)$ and $\HH^k_{\rho^\alpha}(B^d)$.
\item\label{it:selfAdjoint} For every $u, v \in \HH^2_{\rho^\alpha}(B^d)$,
\begin{equation*}
\langle L\psp{\alpha}(u), v\rangle_{\LL^2_{\rho^\alpha}(B^d)} = \langle u, L\psp{\alpha}(v)\rangle_{\LL^2_{\rho^\alpha}(B^d)}.
\end{equation*}
\item\label{it:W12-inclusion} $\HH^1_{\rho^\alpha}(B^d) \subseteq
\HZ_\alpha(B^d)$ with continuous embedding.
\end{enumerate}
\end{proposition}
\begin{proof} Expanding the terms in \eqref{strong-EV} we find that
\begin{equation*}
L\psp{\alpha}(u) = -\rho \, \lapl \varphi + (2\alpha + 1 + d) \, x \cdot \grad \varphi - \sum_{1\leq i<j\leq d} (x_i^2 \partial_j^2 \varphi + x_j^2 \partial_i^2 \varphi - 2 x_i x_j \partial_i \partial_j \varphi).
\end{equation*}
As the coefficients $\rho$, $x$, $x_i^2$, etc.\ above have $\LL^\infty(B^d)$
derivatives of all orders, part \ref{it:map} becomes readily apparent.

By the divergence theorem and the fact that $(x_j \partial_i - x_i \partial_j)(\rho^\alpha) = 0$ part \ref{it:selfAdjoint} is easily seen to be true
if both $u$ and $v$ lie in $\CC^\infty(\overline{B^d})$. As the latter is dense
in $\HH^2_{\rho^\alpha}(B^d)$ (cf.\
\cite[Remark~11.12.(iii)]{Kufner:1985}), the result extends to $u$ and $v$ in
$\HH^2_{\rho^\alpha}(B^d)$ as well.

Let $u \in \HH^1_{\rho^\alpha}(B^d)$. Again from
\cite[Remark~11.12.(iii)]{Kufner:1985} we know that there is a sequence
$(u_n)_{n \in \natural}$ of $\CC^\infty(\overline{B^d})$ functions converging
to $u$ in the norm of that space. As the $\WZ_\alpha(B^d)$ norm is, up to a
positive constant, bounded by the norm of $\HH^1_{\rho^\alpha}(B^d)$, $\lim_{n \to
\infty} u_n = u$ in $\WZ_\alpha(B^d)$ as well, whence per the definition
\eqref{definition-of-HZ}, $u \in \HZ_\alpha(B^d)$ and we have part
\ref{it:W12-inclusion}.
\end{proof}

\begin{lemma}\label{lem:sequences} If $\alpha > -1$ and $k \in \natural_0$,
there exists a positive constant $C = C(\alpha, d, k)$ such that
\begin{equation*}
\left(\forall\, u \in \HH^k_{\rho^\alpha}(B^d)\right) \qquad \sum_{(N,j,\nu) \in \mathcal{I}^d} \left(\lambda_{\alpha,N}\right)^k \abs{\hat u\psp{\alpha}_{N,j,\nu}}^2 h^N_{j,\nu}(\rho^\alpha)
\leq C \norm{u}_{\HH^k_{\rho^\alpha}(B^d)}^2.
\end{equation*}
\end{lemma}
\begin{proof} Let us suppose first that $k$ is even. Then, from parts
\ref{it:map} and \ref{it:selfAdjoint} of \autoref{pro:chainIngredient}, the
fact that the $P^N_{j,\nu}(\rho^\alpha; \placeholder)$ are polynomials (and
thus members of $\HH^k_{\rho^\alpha}(B^d)$) and eigenfunctions of
$L\psp{\alpha}$ and the Parseval identity \eqref{Parseval-0},
\begin{equation*}
C_1 \norm{u}_{\HH^k_{\rho^\alpha}(B^d)}^2
\geq \norm{ (L\psp{\alpha})^{k/2}(u) }_{\LL^2_{\rho^\alpha}(B^d)}^2
 = \sum_{(N,j,\nu) \in \mathcal{I}^d} \abs{ \left(\lambda_{\alpha, N}\right)^{k/2} \hat u\psp{\alpha}_{N, j, \nu} }^2 h^N_{j,\nu}(\rho^\alpha).
\end{equation*}
Now, if $k$ is odd, similarly as above but this time also using part
\ref{it:W12-inclusion} of \autoref{pro:chainIngredient} and the Parseval identity
\eqref{Parseval-1},
\begin{multline*}
C_2 \norm{u}_{\HH^k_{\rho^\alpha}(B^d)}^2
\geq \norm{ (L\psp{\alpha})^{\frac{k-1}{2}}(u) }_{\HH^1_{\rho^\alpha}(B^d)}^2\\
\geq C_3 \left[ \norm{ (L\psp{\alpha})^{\frac{k-1}{2}}(u) }_{\WZ_\alpha(B^d)}^2 - \norm{ (L\psp{\alpha})^{\frac{k-1}{2}}(u) }_{\LL^2_{\rho^\alpha}(B^d)}^2 \right]\\
= C_3 \sum_{(N,j,\nu) \in \mathcal{I}^d} \lambda_{\alpha,N} \abs{ \left(\lambda_{\alpha,N}\right)^{\frac{k-1}{2}} \hat u\psp{\alpha}_{N,j,\nu} }^2 h^N_{j,\nu}(\rho^\alpha).
\end{multline*}
\end{proof}

Given $N \in \natural_0$ let $\proj\psp{\alpha}_N \colon
\LL^2_{\rho^\alpha}(B^d) \to \poly^d_N$ be the orthogonal projection from
$\LL^2_{\rho^\alpha}(B^d)$ onto $\poly^d_N$. On account of \eqref{expansion} in
\autoref{lem:allArePolys} we can express it as a truncation operator:
\begin{equation}\label{projection}
(\forall\,u\in\LL^2_{\rho^\alpha}(B^d)) \quad \proj\psp{\alpha}_N(u) = \sum_{n=0}^N \sum_{(j,\nu) \in \mathcal{I}^d_n} \hat u\psp{\alpha}_{n,j,\nu} \, P^n_{j,\nu}(\rho^\alpha; \placeholder).
\end{equation}

\begin{corollary}\label{cor:L2-approx} If $\alpha > -1$ and $k \in \natural_0$
there exists a positive constant $C = C(\alpha, d, k)$ such that
\begin{equation*}
\left(\forall\, u \in \HH^k_{\rho^\alpha}(B^d)\right) \quad (\forall\, N \in \natural_0) \quad
\norm[n]{u - \proj\psp{\alpha}_N(u)}_{\LL^2_{\rho^\alpha}(B^d)} \leq C (N+1)^{-k} \norm{u}_{\HH^k_{\rho^\alpha}(B^d)}.
\end{equation*}
\end{corollary}
\begin{proof} From \eqref{Parseval-0} in \autoref{lem:allArePolys},
\autoref{lem:sequences} and \eqref{eigenvalues},
\begin{equation*}
\begin{split}
\norm{u - \proj\psp{\alpha}_N(u)}_{\LL^2_{\rho^\alpha}(B^d)}^2
& = \sum_{n=N+1}^\infty \sum_{(j,\nu) \in \mathcal{I}^d_n} \abs{\hat u\psp{\alpha}_{n,j,\nu}}^2 h^n_{j,\nu}(\rho^\alpha)\\
& \leq \sup_{n \geq N+1} \frac{1}{\left(\lambda_{\alpha,n}\right)^k} \sum_{n=N+1}^\infty \sum_{(j,\nu) \in \mathcal{I}^d_n} \left(\lambda_{\alpha,n}\right)^k \abs{\hat u\psp{\alpha}_{n,j,\nu}}^2 h^n_{j,\nu}(\rho^\alpha)\\
& \leq C_1 \left((N + 1)(N + 1 + d +2\alpha)\right)^{-k} \norm{u}_{\HH^k_{\rho^\alpha}(B^d)}^2.
\end{split}
\end{equation*}
Upon taking the square root of both ends of the above chain of inequalities and
using the fact that there exists $C_2 > 0$ such that $(N+1+d+2\alpha)^{-k} \leq
C_2 (N+1)^{-k}$ for $N \in \natural_0$, we obtain the desired result.
\end{proof}

\begin{remark}\label{rem:L2-approx} The result of \autoref{cor:L2-approx} is
essentially a particular case of \cite[Corollary~4.4]{Xu:2005} (which also
encompasses the case of approximation in $\LL^p_{\rho^\alpha}(B^d)$ for general
$p \in [1,\infty]$ and a wider class of weights). The reason why we chose to
present our own proof is because of its simplicity following the intermediate
result \autoref{lem:sequences}, which in turn is needed in the sequel. At this
stage it is relevant to point out that the same result could be obtained with
spaces that, with respect to the $\HH^k_{\rho^\alpha}(B^d)$ spaces, have a
slightly less stringent requirement on the radial derivative
of its members but still map to $\LL^2_{\rho^\alpha}(B^d)$ ($k$ even) or $\HH^1_{\rho^\alpha}(B^d)$ ($k$ odd) under the
action of the $\floor{k/2}$-th power of the operator $L\psp{\alpha}$. In the one-dimensional case this idea is pursued in \cite{Nicaise}.
\end{remark}

The orthogonal projection operators defined in \eqref{projection} allow for
defining an equivalent norm for the $\HH^k_{\rho^\alpha}(B^d)$ spaces.

\begin{proposition}\label{pro:equivalentNorm} Let $\alpha > -1$ and $k \in \natural$.
Then, the functional $u \mapsto \abs{u}_{\HH^k_{\rho^\alpha}(B^d)} +
\norm[n]{\proj\psp{\alpha}_{k-1}(u)}_{\LL^2_{\rho^\alpha}(B^d)}$ is an
equivalent norm for $\HH^k_{\rho^\alpha}(B^d)$.
\end{proposition}
\begin{proof}
Setting $\Omega = B^d$, $\kappa = 1$, $p = q = 2$, $\beta = \alpha$ and $\alpha
= \alpha$ in Theorem 8.8 of \cite{OG:1989} we have that
$\HH^1_{\rho^\alpha}(B^d) \compEmb \LL^2_{\rho^\alpha}(B^d)$. By standard
arguments \cite[Remark~6.4.4]{AF:2003} this implies that
$\HH^k_{\rho^\alpha}(B^d) \compEmb \HH^{k-1}_{\rho^\alpha}(B^d)$. Then, the
desired result follows from the Peetre--Tartar lemma; in the formulation of
\cite[Lemma~11.1]{Tartar:2007} it comes from setting $E_1 =
\HH^k_{\rho^\alpha}(B^d)$, $E_2 =
[\LL^2_{\rho^\alpha}(B^d)]^{\binom{k+d-1}{k}}$, $E_3 =
\HH^{k-1}_{\rho^\alpha}(B^d)$, $A = \grad_k$, $B$ the injection from
$\HH^k_{\rho^\alpha}(B^d)$ onto $\HH^{k-1}_{\rho^\alpha}(B^d)$, $G =
\LL^2_{\rho^\alpha}(B^d)$ and $M = \proj\psp{\alpha}_{k-1}$ and noting that
$\nabla_k u \equiv 0$ implies $u \in \poly^d_{k-1}$, which in turn is a
consequence of \cite[Lemma~6.4]{Tartar:2007}.
\end{proof}

\subsection{Inverse or Markov-type inequality}

We will later have use of a Markov-type inequality (\autoref{lem:Markov} below).
We were not able to reproduce the very direct proofs that work in the one dimensional case (cf.\ \cite[p.~298]{CHQZ-I} and \cite[Theorem 2.3]{Guo:2000a}), so had to take a detour through Remez- and Bernstein-type inequalities.
Our argument roughly corresponds to some in \cite{Ditzian:1992} where the unweighted case is treated for more general geometry and for general $\LL^p$ norms.
Other related integral inequalities for polynomials appear in \cite{Dai:2006} and \cite{BottcherDorfler:2011} and are reported in \cite{Daugavet:1972}.

\begin{proposition}\label{pro:Remez} Let $\alpha > -1$. If the positive number
$A$ is small enough, then there exists $C = C(A, \alpha, d) > 0$ such that
\begin{equation}\label{Remez}
(\forall \, N \in \natural)\ (\forall\,p\in\poly^d_N) \quad
\int_{B^d} \abs{p(x)}^2 \rho(x)^\alpha \dd x
\leq C \int_{B_{\Real^d}\left(0,1-A N^{-2}\right)} \abs{p(x)}^2 \rho(x)^\alpha \dd x.
\end{equation}
\end{proposition}
\begin{proof} We know from equations 2.3 (weights of Jacobi type are doubling)
and 7.17 (Remez-type inequality for doubling weights and algebraic polynomials)
of \cite{MT:2000} that the following Remez-type inequality holds in the
one-dimensional case: Let $\alpha, \beta > -1$. Then, there exists $\Lambda_0 =
\Lambda_0(\alpha,\beta) > 0$ such that for every $\Lambda \leq \Lambda_0$ there
exists $C_{\mathrm{R}} = C_{\mathrm{R}}(\alpha,\beta,\Lambda) > 0$ such that,
in turn, for every $N \in \natural$ and every $p \in \poly^1_N$,
\begin{equation}\label{Remez-1d}
\int_{-1}^1 \abs{p(x)}^2 (1-x)^\alpha (1+x)^\beta \dd x \leq C_{\mathrm{R}} \int_{-(1-\Lambda N^{-2})}^{1-\Lambda N^{-2}} \abs{p(x)}^2 (1-x)^\alpha (1+x)^\beta \dd x.
\end{equation}
So, setting $\beta = \alpha$, \eqref{Remez} is true in the $d = 1$ case.

Let $N \in \natural$ and $p \in \poly^d_N$. Then, the orthogonal expansion
\eqref{expansion} of $p$ is finite; using \eqref{BOVN}:
\begin{equation*}
p(x) = \sum_{k=0}^N \sum_{(j,\nu) \in \mathcal{I}^d_k} \hat p\psp{\alpha}_{k,j,\nu} \, Y^d_{k-2j,\nu}(x) \Jac[j]{\alpha}{k-2j+\frac{d-2}{2}}(2 \abs{x}^2 - 1).
\end{equation*}
Because of the form the index sets $\mathcal{I}^d_k$ have (cf.\
\eqref{degreeNIndices}) we can rearrange the sum as
\begin{equation*}
p(x) = \sum_{l = 0}^N \sum_{\nu = 1}^{\dim(\mathcal{H}^d_l)} Y^d_{l,\nu}(x) \, \hat p\psp{\alpha}_{\nu,l}(2\abs{x}^2-1)
\quad\text{where}\quad
\hat p\psp{\alpha}_{\nu,l} = \sum_{j=0}^{\floor{(N-l)/2}} \hat p\psp{\alpha}_{l+2j,j,\nu}
\Jac[j]{\alpha}{l+\frac{d-2}{2}} \in \poly^d_{\floor{(N-l)/2}}.
\end{equation*}
Using the
generalized polar integration formula \eqref{polarIntegral} and the fact that
the $Y^d_{l,\nu}$ are homogeneous polynomials of degree $l$ (whence
$Y^d_{l,\nu}(r y) = r^l Y^d_{l,\nu}(y)$ for all $r > 0$ and $y \in
\mathbb{S}^{d-1}$) and pairwise $\LL^2(\sigma_{d-1})$-orthogonal,
\begin{equation}\label{Remez-reduction-to-1d}
\int_{B^d} \abs{p(x)}^2 \rho(x)^\alpha \dd x
 = \sum_{l = 0}^N \sum_{\nu = 1}^{\dim(\mathcal{H}^d_l)} \norm{Y^d_{l,\nu}}_{\LL^2(\sigma_{d-1})}^2 \int_0^1 \abs{r^l \hat p\psp{\alpha}_{\nu,l}(2 r^2-1)}^2 (1-r^2)^\alpha \, r^{d-1} \dd r.
\end{equation}
Each of the functions $r \mapsto r^l \, \hat p\psp{\alpha}_{\nu,l}(2 r^2 - 1)$ is a
polynomial of degree less than or equal to $N$. Performing the change of
variable $r' = 2r-1$ they remain so and $r \mapsto (1-r^2)^\alpha r^{d-1}$
turns into a function which in $(-1, 1)$ is bounded from above and below by
positive constants times $r' \mapsto (1-r')^\alpha (1+r')^{d-1}$. Hence, we can
appeal to \eqref{Remez-1d} with the same $\alpha$ and $\beta = d-1$ and
performing the inverse change of variable we obtain that for $A$ small enough
there exists $C > 0$ depending on $A$, $\alpha$ and the dimension $d$ such that
\begin{equation*}
\int_0^1 \abs{r^l \hat p\psp{\alpha}_{\nu,l}(2 r^2-1)}^2 (1-r^2)^\alpha \, r^{d-1} \dd r
\leq C \int_{A N^{-2}}^{1-A N^{-2}} \abs{r^l \hat p\psp{\alpha}_{\nu, l}\left(2 r^2 - 1 \right)}^2 (1-r^2)^\alpha r^{d-1} \dd r.
\end{equation*}
Using this in \eqref{Remez-reduction-to-1d} and using again the
$\LL^2(\sigma_{d-1})$ orthogonality of the $Y^d_{l,\nu}$ and the generalized
polar integration formula \eqref{polarIntegral} we obtain that $\int_{B^d}
\abs{p(x)}^2 \rho(x)^\alpha \dd x$ can be bounded by $C$ times the same
integral taken over $\{x \in \Real^d \mid A N^{-2} < \abs{x} < 1-A N^{-2}\}$ and
this implies the desired result.
\end{proof}

\begin{lemma}\label{lem:Markov} Let $\alpha > -1$. Then there exists $C =
C(\alpha,d) > 0$ such that for all $N \in \natural_0$ and $p \in \poly^d_N$,
\begin{equation*}
\norm{\grad p}_{[\LL^2_{\rho^\alpha}(B^d)]^d} \leq C N^2 \norm{p}_{\LL^2_{\rho^\alpha}(B^d)}.
\end{equation*}
\end{lemma}
\begin{proof} We start by proving a Bernstein-type inequality. Let $N \in
\natural_0$ and let $q \in \poly^d_N$ so its expansion according to
\eqref{expansion} is finite: $q = \sum_{k=0}^N \sum_{(j,\nu) \in
\mathcal{I}^d_k} \hat q\psp{\alpha}_{k,j,\nu}
P^k_{j,\nu}(\rho^\alpha;\placeholder)$. Then, as the
$P^k_{j,\nu}(\rho^\alpha;\placeholder)$ are eigenfunctions of the problem
\eqref{wrong-EV-problem} (cf.\ \autoref{lem:allArePolys})---with associated
eigenvalue $\lambda_{\alpha,k}$ which is a monotone increasing function of $k$
(cf.\ \eqref{eigenvalues})---and pairwise
$\LL^2_{\rho^\alpha}(B^d)$-orthogonal,
\begin{multline}\label{Bernstein}
\norm{\grad q}_{[\LL^2_{\rho^{\alpha+1}}(B^d)]^d}^2
\leq \norm{\grad q}_{[\LL^2_{\rho^{\alpha+1}}(B^d)]^d}^2 + \sum_{1 \leq i < j \leq d} \norm{x_j \partial_i q - x_i \partial_j q}_{\LL^2_{\rho^\alpha}(B^d)}^2\\
= \sum_{k=0}^N \sum_{(j,\nu) \in \mathcal{I}^d_k} \abs{q\psp{\alpha}_{k,j,\nu}}^2 \lambda_{\alpha,k} \norm{P^k_{j,\nu}(\rho^\alpha;\placeholder)}_{\LL^2_{\rho^\alpha}(B^d)}^2
\leq N (N + d + 2\alpha) \norm{q}_{\LL^2_{\rho^\alpha}(B^d)}^2.
\end{multline}

Let us now consider $p \in \poly^d_N$ with $N \geq 2$ (our desired result is
obviously true in the case $N = 0$ with any $C$ and the case $N = 1$ can be
incorporated later at the possible price of an enlargement of $C$). Given any
direction $i \in \{1, \dotsc, N\}$, $\partial_i p \in \poly^d_{N-1}$. From
\autoref{pro:Remez} there exist positive numbers $A$ and $C_1$ such that
\begin{multline*}
\int_{B^d} \abs{\partial_i p(x)}^2 \rho(x)^\alpha \dd x
\leq C_1 \int_{B_{\Real^d}\left(0,1-A(N-1)^{-2}\right)} \abs{\partial_i p(x)}^2 \rho(x)^\alpha \dd x\\
\leq C_1 \left[\sup_{B_{\Real^d}\left(0,1-A(N-1)^{-2}\right)} \rho^{-1} \right] \int_{B^d} \abs{\partial_i(x)}^2 \rho(x)^{\alpha+1} \dd x.
\end{multline*}
We can replace $\rho^{-1}$ with $x \mapsto (1-\abs{x})^{-1}$ in the supremum
above which then evaluates to $(N-1)^2/A$ to obtain an upper bound for the last
expression. Summing the resulting inequality with respect to $i$ and using the
Bernstein-type inequality \eqref{Bernstein},
\begin{equation*}
\norm{\grad p}_{[\LL^2_{\rho^\alpha}(B^d)]^d}^2 \leq C_1/A (N-1)^2 N (N+2+2\alpha) \norm{p}_{\LL^2_{\rho^\alpha}(B^d)}^2
\end{equation*}
which, through another possible worsening of the constant, implies the desired result.
\end{proof}

\section{Truncation projection in the two-dimensional case}\label{sec:2d}

\subsection{Zernike polynomials and Fourier--Zernike series}\label{ssc:ZPZFS}

Let $\theta \colon B^2 \rightarrow \Real$ and $r \colon B^2 \rightarrow \Real$
be the usual components of the Cartesian-to-polar change of coordinates. Then,
an admissible basis of $\mathcal{H}^2_n$ is readily identified as $\{1\}$ if $n
= 0$ and $\{ r^n \cos(n \theta) , r^n \sin(n \theta) \}$ if $n \geq 1$. Yet
instead of using the resulting system of solutions of \eqref{wrong-EV-problem}
exactly as given in \eqref{BOVN}, we will find it more convenient to use the
following recombined, re-indexed and rescaled form found in
\cite[eq.~2.1]{Wunsche:2005}:
\begin{equation}\label{ev-polar-separation}
(\forall\,(m,n) \in \natural_0 \times \natural_0) \quad P\psp{\alpha}_{m,n}
= \frac{\Gamma(\min(m,n)+1) \Gamma(\alpha+1)}{\Gamma(\min(m,n) + \alpha + 1)}
r^{\abs{m-n}} e^{\imath (m-n) \theta}
\Jac[\min(m,n)]{\alpha}{\abs{m-n}}(2r^2-1).
\end{equation}
In order to simplify some expressions below we adopt the convention
\begin{equation}\label{convention}
P\psp{\alpha}_{m,n} \equiv 0 \quad \text{if } m < 0 \text{ or } n < 0.
\end{equation}
As in the above re-indexing $\abs{m-n}$ is the degree of the spherical harmonic
and $\min(m,n)$ that of the Jacobi polynomial involved, the degree of the
polynomial $P\psp{\alpha}_{m,n}$ is $\abs{m-n} + 2 \min(m,n) = m + n$, whence
its associated eigenvalue is (cf.\ \eqref{eigenvalues})
\begin{equation}\label{eigenvalues-2d}
\lambda\psp{\alpha}_{m,n} = (m + n) (m + n + 2 + 2 \alpha).
\end{equation}
The attractiveness of the precise form for the basis functions given in
\eqref{ev-polar-separation} is apparent in the light of the simplicity of the
relations (cf.\ equations 3.4 and 5.3 of \cite{Wunsche:2005})
\begin{gather}
\label{squaredNorm-2d}
h\psp{\alpha}_{m,n} := \norm[n]{P\psp{\alpha}_{m,n}}_{\LL^2_{\rho^\alpha}(B^2)}^2
= \frac{\pi \, \Gamma(\alpha + 1)^2}{m + n + \alpha + 1} \frac{\Gamma(m + 1)}{\Gamma(m + \alpha + 1)} \frac{\Gamma(n + 1)}{\Gamma{(n + \alpha + 1)}},\\
\label{rotated-diffs}
\partial_{z^*} P\psp{\alpha}_{m,n} = \frac{(m + \alpha + 1) n}{\alpha + 1} P\psp{\alpha+1}_{m, n-1}
\qquad\text{and}\qquad
\partial_z P\psp{\alpha}_{m,n} = \frac{m (n + \alpha + 1)}{\alpha + 1} P\psp{\alpha+1}_{m-1,n},
\end{gather}
which are valid for all $(m,n) \in \natural_0 \times \natural_0$; here,
$\partial_{z^*} = \frac{1}{2} \left(\partial_1 + \imath\,\partial_2\right) =
\frac{e^{\imath \theta}}{2} \left( \partial_r + \frac{\imath}{r}
\partial_\theta \right)$ and $\partial_z = \frac{1}{2} \left(\partial_1 -
\imath \, \partial_2\right) = \frac{e^{-\imath \theta}}{2} \left( \partial_r -
\frac{\imath}{r} \partial_\theta \right)$. For analytical purposes these
differential operators can be used in lieu of the canonical ones because for
all weakly differentiable $u$ and $l \in \natural$, the relation
\begin{equation}\label{proto-equivalent-seminorm}
\abs{\grad_l u}^2 = \sum_{l_1+l_2=l}\abs[n]{\partial_1^{l_1} \partial_2^{l_2} u}^2 \cong_l \sum_{l_1+l_2=l} \abs[n]{\partial_z^{l_1} \partial_{z^*}^{l_2} u}^2
\end{equation}
holds almost everywhere, with $\cong_l$ meaning that each side is bounded by
the other times some positive constant depending on $l$ only. When $l = 1$ the
left-hand side is exactly twice the right-hand side almost everywhere.

We now translate some of the results of \autoref{sec:PEEUB} to the reindexed
and rescaled basis \eqref{ev-polar-separation}. Suppose that $\alpha > -1$.
From \autoref{lem:allArePolys},
\begin{equation}\label{expansion-2d}
\left(\forall\, u \in \LL^2_{\rho^\alpha}(B^2)\right) \qquad u = \sum_{(m,n) \in \natural_0 \times \natural_0} \hat u\psp{\alpha}_{m,n} \, P\psp{\alpha}_{m,n}
\end{equation}
in the $\LL^2_{\rho^\alpha}(B^2)$ sense in general and in the $\WZ_\alpha(B^2)$
sense if, in addition, $u \in \HZ_\alpha(B^2)$; here, for all $u \in
\LL^2_{\rho^\alpha}(B^2)$ and $(m,n) \in \natural_0 \times \natural_0$,
\begin{equation}\label{coefficient-2d}
\hat u\psp{\alpha}_{m,n} := \left\langle u, P\psp{\alpha}_{m,n}\right\rangle_{\LL^2_{\rho^\alpha}(B^2)} \left/ h\psp{\alpha}_{m,n} \right. .
\end{equation}
Further, Parseval's identity manifests itself as
\begin{equation}\label{Parseval-2d}
\begin{cases}
(\forall\,u\in\LL^2_{\rho^\alpha}(B^2)) \quad \norm{u}_{\LL^2_{\rho^\alpha}(B^2)}^2\\
(\forall\,u\in\HZ_\alpha(B^2)) \quad \norm{u}_{\WZ_\alpha(B^2)}^2
\end{cases}
= \sum_{(m,n) \in \natural_0 \times \natural_0}
\begin{cases}
1\\
1 + \lambda\psp{\alpha}_{m,n}
\end{cases}
\abs{\hat u\psp{\alpha}_{m,n}}^2 h\psp{\alpha}_{m,n}.
\end{equation}
From \autoref{lem:sequences} we know that there exists a positive constant $C =
C(\alpha,k)$ such that
\begin{equation}\label{sequences-2d}
\left(\forall\, u \in \HH^k_{\rho^\alpha}(B^2)\right) \qquad \sum_{(m,n) \in \natural_0 \times \natural_0} \left(\lambda\psp{\alpha}_{m,n}\right)^k \abs{\hat u\psp{\alpha}_{m,n}}^2 h\psp{\alpha}_{m,n}
\leq C \norm{u}_{\HH^k_{\rho^\alpha}(B^2)}^2.
\end{equation}
The projection (truncation) operator $\proj\psp{\alpha}_N \colon
\LL^2_{\rho^\alpha}(B^2) \rightarrow \poly^2_N$ of \eqref{projection} here takes the form
\begin{equation}\label{projector-2d}
(\forall\,u\in\LL^2_{\rho^\alpha}(B^2)) \quad \proj\psp{\alpha}_N(u) = \sum_{m+n \leq N} \hat u\psp{\alpha}_{m,n} \, P\psp{\alpha}_{m,n}.
\end{equation}

\begin{proposition}[Connection coefficients between Zernike
polynomials]\label{pro:Zernike-CC} If $\alpha, \gamma > -1$ and $(m, n) \in
\natural_0 \times \natural_0$,
\begin{multline*}
P\psp{\alpha}_{m,n} = \frac{\Gamma(m+1) \Gamma(n+1) \Gamma(\alpha+1)}{\Gamma(\alpha + m + 1) \Gamma(\alpha + n + 1) \Gamma(\gamma+1)} \sum_{k=0}^{\min(m,n)} \Bigg[ \frac{\poch{\alpha-\gamma}{k} \Gamma(\alpha+m+n-k+1)}{\Gamma(k+1)}\\
\times \frac{\Gamma(\gamma+m-k+1) \Gamma(\gamma+n-k+1) (\gamma+m+n-2k+1)}{\Gamma(m-k+1) \Gamma(n-k+1) \Gamma(\gamma + m + n - k + 2)} \Bigg]
P\psp{\gamma}_{m-k,n-k}.
\end{multline*}
\end{proposition}
\begin{proof} From the definition \eqref{ev-polar-separation} of
$P\psp{\alpha}_{m,n}$, using \eqref{Jac-from-AAR} to expand
$P\psp{\alpha,\abs{m-n}}_{\min(m,n)}$ in terms of $P\psp{\gamma,\abs{m-n}}_j$,
$j \in \{0, \dotsc, \min(m,n)\}$, expanding Pochhammer symbols with suitable
arguments into ratios of gamma functions, using the basic property $x \,
\Gamma(x) = \Gamma(x+1)$, the fact that $\min(m,n) + \abs{m-n} = \max(m,n)$,
the fact that for any commutative function $B \colon \natural_0 \times
\natural_0 \to \Real$ it holds that $B(\min(m,n),\max(m,n)) = B(m,n)$ and some
cancellations, we find that
\begin{multline*}
P\psp{\alpha}_{m,n}
= \frac{\Gamma(m+1) \Gamma(n+1) \Gamma(\alpha + 1)}{\Gamma(\alpha + m + 1) \Gamma(\alpha + n + 1)}
\sum_{j=0}^{\min(m,n)} \frac{\poch{\alpha-\gamma}{\min(m,n)-j} \Gamma(\alpha + \max(m,n) + 1 + j)}{\Gamma(\min(m,n)-j+1)}\\
\times \frac{ \Gamma(\gamma + \abs{m-n} + 1 + j) (\gamma + \abs{m-n} + 2j + 1) }{ \Gamma(\abs{m-n} + 1 + j) \Gamma(\gamma + \max(m,n) + 2 +j) }
r^{\abs{m-n}} e^{\imath (m-n) \theta} \Jac[j]{\gamma}{\abs{m-n}}(2r^2-1).
\end{multline*}
On defining $m_j := j + \max(m-n,0)$ and $n_j := j + \max(n-m,0)$ and noting
that
\begin{equation*}
m_j \geq 0, \quad n_j \geq 0, \quad m - n = m_j - n_j \quad\text{and}\quad j = \min(m_j, n_j),
\end{equation*}
we find that dividing and multiplying each term of the above sum by
$\frac{\Gamma(j+1) \Gamma(\gamma+1)}{\Gamma(j + \gamma + 1)}$ we can make
$P\psp{\gamma}_{m_j,n_j}$ appear. Substituting the summation variable for $k =
\min(m,n) - j$ (wherein $m_j$ and $n_j$ turn into $m-k$ and $n-k$,
respectively) and using $\abs{m-n} + 2 \min(m,n) = m + n$ plus some of the
previously used identities we obtain the desired result after a number of
elementary cancellations.
\end{proof}

We can now deduce some simple relations between Zernike polynomials which will be useful later to express the expansion coefficients of the derivatives of a function in terms of the expansions coefficients of the function itself.
Related identities including three term recurrences appear in \cite[\S~5]{Wunsche:2005}; \eqref{dxpidy} and \eqref{idxpdy} appear in \cite{Janssen:2014} in the case $\alpha = 0$.

\begin{proposition}\label{pro:extra-expansions} If $\alpha > -1$, then for
$(m,n) \in \natural_0 \times \natural_0$, we have the parameter-raising
expansion
\begin{equation}\label{parameter-raise}
(m + n + \alpha + 1) P\psp{\alpha}_{m,n} = \frac{(m + \alpha + 1) (n + \alpha + 1)}{\alpha + 1} P\psp{\alpha+1}_{m,n} - \frac{m \, n}{\alpha + 1} P\psp{\alpha+1}_{m-1, n-1},
\end{equation}
and the same-parameter expansions with respect to first order derivatives
\begin{equation}\label{dxpidy}
(m + n + \alpha + 1) P\psp{\alpha}_{m,n}
= \frac{n + \alpha + 1}{n + 1} \partial_{z^*} P\psp{\alpha}_{m,n+1} - \frac{m}{m + \alpha} \partial_{z^*} P\psp{\alpha}_{m-1, n}
\end{equation}
and
\begin{equation}\label{idxpdy}
(m + n + \alpha + 1) P\psp{\alpha}_{m,n}
= \frac{m + \alpha + 1}{m + 1} \partial_z P\psp{\alpha}_{m+1,n} - \frac{n}{n + \alpha} \partial_z P\psp{\alpha}_{m,n-1}.
\end{equation}
\end{proposition}
\begin{proof} We obtain \eqref{parameter-raise} from \autoref{pro:Zernike-CC}
by setting $\gamma = \alpha + 1$. Combining \eqref{parameter-raise} with
adequate shifts of the relations in \eqref{rotated-diffs} yields \eqref{dxpidy}
and \eqref{idxpdy}.
\end{proof}

\begin{proposition}\label{pro:decayRate} Let $u \in
\HH^k_{\rho^\alpha}(B^2)$. Then,
\begin{equation*}
(\forall\,(m,n) \in \natural_0 \times \natural_0) \quad
\lim_{L \to \infty} L^{k-\alpha-1/2} \, \hat u\psp{\alpha}_{m+L,n+L} = 0.
\end{equation*}

\end{proposition}
\begin{proof} Using the forms of $\lambda\psp{\alpha}_{m+L,n+L}$ and
$h\psp{\alpha}_{m+L,n+L}$ which stem from \eqref{eigenvalues-2d} and
\eqref{squaredNorm-2d}, respectively, and applying the asymptotic formula
\eqref{GRA} on the ratio of gamma functions therein we obtain
\begin{equation*}
\lambda\psp{\alpha}_{m+L,n+L} \sim 4 L^2
\qquad\text{and}\qquad
h\psp{\alpha}_{m+L,n+L} \sim \pi \, \Gamma(\alpha+1)^2 \, 2^{-1} L^{-1-2\alpha}
\qquad\text{as } L \to \infty.
\end{equation*}
Combining this with the fact (coming from \eqref{sequences-2d}) that
\begin{equation*}
\lim_{L \to \infty} (\lambda\psp{\alpha}_{m+L,n+L})^k \, \abs[n]{u\psp{\alpha}_{m+L,n+L}}^2 \,
h\psp{\alpha}_{m+L,n+L} = 0
\end{equation*}
we obtain the desired result.
\end{proof}

\begin{lemma}\label{lem:coefs-of-diffs} Let $\alpha > -1$ and
\begin{equation}\label{coefs-of-diff-hyp}
u \in \HH^k_{\rho^\alpha}(B^2)
\quad\text{with}\quad
k = \begin{cases} 1 & \text{if } \alpha \in [-1/2, \infty),\\ 2 & \text{if } \alpha \in (-1, -1/2). \end{cases}
\end{equation}
Then, the
coefficients of the Fourier--Zernike series \eqref{expansion-2d} of
$\partial_{z^*} u$ and $\partial_z u$ can be expressed in terms of the coefficients
of the corresponding series of $u$ according to
\begin{align}
\label{coefs-of-diff-1}
\widehat{\left(\partial_{z^*} u\right)}\psp{\alpha}_{m,n} & = (m + n + \alpha + 1) \sum_{l=0}^\infty \frac{\poch{m+1}{l}}{\poch{m+\alpha+1}{l}} \frac{\poch{n+1}{l+1}}{\poch{n+\alpha+1}{l+1}} \, \hat u\psp{\alpha}_{m+l, n+1+l}
\intertext{and}
\label{coefs-of-diff-2}
\widehat{\left(\partial_z u\right)}\psp{\alpha}_{m,n} & = (m + n + \alpha + 1) \sum_{l=0}^\infty \frac{\poch{m+1}{l+1}}{\poch{m+\alpha+1}{l+1}} \frac{\poch{n+1}{l}}{\poch{n+\alpha+1}{l}} \, \hat u\psp{\alpha}_{m+1+l, n+l}
\end{align}
for $(m, n) \in \natural_0 \times \natural_0$.
\end{lemma}
\begin{proof}
Let us abbreviate $v_{m,n} =
\widehat{\left(\partial_{z^*} u\right)}\psp{\alpha}_{m,n}$. Then, using
\eqref{dxpidy}, the fact that $P\psp{\alpha}_{-1,n} \equiv 0$ and $\partial_{z^*}
P\psp{\alpha}_{m,0} \equiv 0$ (cf.\ \eqref{convention} and
\eqref{rotated-diffs}) and careful index tracking we have that given $M, N \in
\natural$,
\begin{multline}\label{threeParts}
\sum_{m=0}^M \sum_{n=0}^N v_{m,n} \, P\psp{\alpha}_{m,n}
= \sum_{m=0}^{M-1} \sum_{n=1}^N \left[ \frac{v_{m,n-1}}{m+n+\alpha} \frac{n+\alpha}{n} - \frac{v_{m+1,n}}{m+n+\alpha+2} \frac{m+1}{m+\alpha+1} \right] \partial_{z^*} P\psp{\alpha}_{m,n}\\
+ \underbrace{\sum_{m=0}^M \frac{v_{m,N}}{m+N+\alpha+1} \frac{N+\alpha+1}{N+1} \partial_{z^*} P\psp{\alpha}_{m,N+1}}_{:= S_{M,N}}
+ \underbrace{\sum_{n=0}^{N-1} \frac{v_{M,n}}{M+n+\alpha+1} \frac{n+\alpha+1}{n+1} \partial_{z^*} P\psp{\alpha}_{M,n+1}}_{:= T_{M,N}}.
\end{multline}
Taking the square of the $\LL^2_{\rho^\alpha}(B^2)$ norm of $S_{M,N}$, using
the $\LL^2_{\rho^{\alpha+1}}(B^2)$-orthogonality of the terms that comprise it
(which comes from \eqref{rotated-diffs}), substituting the resulting
$h\psp{\alpha+1}_{m,N}$ with the products $h\psp{\alpha}_{m,N} \left(
h\psp{\alpha+1}_{m,N} / h\psp{\alpha}_{m,N} \right)$ and simplifying the gamma
functions appearing in the second factors (cf.\ \eqref{squaredNorm-2d}) we
obtain
\begin{equation*}
\norm{S_{M,N}}_{\LL^2_{\rho^{\alpha+1}}(B^2)}^2
= \sum_{m=0}^M \abs{v_{m,N}}^2 \frac{(N+\alpha+1) (m+\alpha+1)}{(m+N+\alpha+1) (m+N+\alpha+2)} \, h\psp{\alpha}_{m,N}
\leq \sum_{m=0}^M \abs{v_{m,N}}^2 \, h\psp{\alpha}_{m,N}.
\end{equation*}
As the $v_{m,N}$ are the coefficients of the expansion of the
$\LL^2_{\rho^\alpha}(B^2)$ function $\partial_{z^*} u$, it follows from the above
inequality and Parseval's identity \eqref{Parseval-2d} that $S_{M,N}
\xrightarrow{M,N \to \infty} 0$ in $\LL^2_{\rho^{\alpha+1}}(B^2)$. The
same argument leads to $T_{M,N} \xrightarrow{M,N \to \infty} 0$ in the same
space. As the left hand side of \eqref{threeParts} tends to $\partial_{z^*} u$ as $M, N \to \infty$ in $\LL^2_{\rho^{\alpha+1}}(B^2)$ (because it does so in the stronger $\LL^2_{\rho^\alpha}(B^2)$ norm) we conclude that
\begin{equation}\label{difficult-expansion}
\partial_{z^*} u = \sum_{m=0}^\infty \sum_{n=1}^\infty \left[ \frac{v_{m,n-1}}{m+n+\alpha} \frac{n+\alpha}{n} - \frac{v_{m+1,n}}{m+n+\alpha+2} \frac{m+1}{m+\alpha+1} \right] \partial_{z^*} P\psp{\alpha}_{m,n},
\end{equation}
the series converging in the $\LL^2_{\rho^{\alpha+1}}(B^2)$ sense.

On the other hand, per part \ref{it:W12-inclusion} of
\autoref{pro:chainIngredient}, $u$ itself is a member of $\HZ_\alpha(B^2)$ and
thus, per part \ref{it:eigenpairs-consequences} of \autoref{lem:allArePolys}, its Fourier--Zernike series as defined in \eqref{expansion-2d} converges
to $u$ in $\WZ_\alpha(B^2)$ and because of the structure of that norm (cf.\
\eqref{definition-of-WZ}) we have, again using the fact that $\partial_{z^*}
P\psp{\alpha}_{m,0} \equiv 0$,
\begin{equation}\label{easy-expansion}
\partial_{z^*} u = \sum_{m=0}^\infty \sum_{n=1}^\infty \hat u\psp{\alpha}_{m,n} \, \partial_{z^*} P\psp{\alpha}_{m,n},
\end{equation}
the series also converging in the $\LL^2_{\rho^{\alpha+1}}(B^2)$ sense.

As in the index range involved the $\partial_{z^*} P\psp{\alpha}_{m,n}$ are
non-zero and pairwise $\LL^2_{\rho^{\alpha+1}}(B^2)$-orthogonal, we can compare
the coefficents of the series \eqref{difficult-expansion} and
\eqref{easy-expansion} so as to obtain for $m \in \natural_0$ and $n \in
\natural_0$,
\begin{equation}\label{link}
\frac{1}{m+n+\alpha+1} v_{m,n}
= \frac{n+1}{n+\alpha+1} \hat u\psp{\alpha}_{m,n+1} + \frac{(n+1) (m+1)}{(m+n+\alpha+3) (n+\alpha+1) (m+\alpha+1)} v_{m+1,n+1}.
\end{equation}
An induction argument based on \eqref{link} can then justify that, for all
$(m,n) \in \natural_0 \times \natural_0$ and $L \in \natural_0$,
\begin{equation}\label{with-residual}
\frac{1}{m+n+\alpha+1} \, v_{m,n}
= \sum_{l=0}^L \frac{\poch{n+1}{l+1}}{\poch{n+\alpha+1}{l+1}} \frac{\poch{m+1}{l}}{\poch{m+\alpha+1}{l}} \, \hat u\psp{\alpha}_{m+l, n+1+l} + R_{m,n,\alpha,L+1}
\end{equation}
where
\begin{equation*}
R_{m,n,\alpha,L} := \frac{\poch{n+1}{L}}{\poch{n+\alpha+1}{L}} \frac{\poch{m+1}{L}}{\poch{m+\alpha+1}{L}} \frac{1}{m+n+2L+\alpha+1} \, v_{m+L, n+L}.
\end{equation*}
Now, expressing the Pochhammer symbols above as ratios of gamma functions and
using the asymptotic relation \eqref{GRA} we find
\begin{equation}\label{asymptotic-R}
R_{m,n,\alpha,L} \sim \frac{\Gamma(n+\alpha+1)}{\Gamma(n+1)} \frac{\Gamma(m+\alpha+1)}{\Gamma(m+1)} \, 2^{-1} L^{-1 -2 \alpha} \, v_{m+L,n+L}
\quad\text{as } L \to \infty.
\end{equation}
So far, we have only used that $u \in \HH^1_{\rho^\alpha}(B^2)$, which is
weaker than the hypothesis \eqref{coefs-of-diff-hyp} when $\alpha \in (-1,
-1/2)$.

As the $v_{m,n}$ are the Fourier--Zernike coefficients of the expansion of
$\partial_{z^*} u$ and the latter belongs to $\HH^{k-1}_{\rho^\alpha}(B^2)$ we
infer from \autoref{pro:decayRate} that $\lim_{L\to\infty} L^{k-\alpha-3/2}
v_{m+L,n+L} = 0$, which, together with the fact that $k \geq 1/2-\alpha$ (only
now we are making use of the full hypothesis \eqref{coefs-of-diff-hyp}) implies
that $\lim_{L \to \infty} R_{m,n,\alpha,L} = 0$. Thus, \eqref{coefs-of-diff-1}
is obtained from \eqref{with-residual}.

Let the reflection $A \colon B^2 \to B^2$ be defined by $A(x) = (x_1, -x_2)$ for all $x \in
B^2$. Then, $u \circ A \in
\HH^k_{\rho^\alpha}(B^2)$ as well and $\partial_z u \circ A = \partial_{z^*} (u \circ
A)$. This, together with \eqref{coefs-of-diff-1}, the readily verifiable
formulae $\rho \circ A = \rho$, $P\psp{\alpha}_{m,n} \circ A =
P\psp{\alpha}_{n,m}$ and $h\psp{\alpha}_{m,n} = h\psp{\alpha}_{n,m}$ and the
invariance of the Lebesgue measure with respect to reflections give
\eqref{coefs-of-diff-2}.
\end{proof}

The hypothesis $u \in \HH^2_{\rho^\alpha}(B^2)$ adopted in
\autoref{lem:coefs-of-diffs} when $\alpha \in (-1, -1/2)$ can be relaxed to $u$
belonging to certain interpolation spaces between
$\HH^1_{\rho^\alpha}(B^2)$ and $\HH^2_{\rho^\alpha}(B^2)$ as long as
the residual $R_{m,n,\alpha,L+1}$ of \eqref{with-residual} can be shown to tend
to $0$ as $L \to \infty$. However, the example below makes it clear that we
cannot relax the hypothesis all the way to the hypothesis $u \in
\HH^1_{\rho^\alpha}(B^2)$ befitting the case in which $\alpha \in [-1/2,
\infty)$.

\begin{proposition}\label{pro:W12-not-enough} Let $\alpha \in (-1, -1/2)$.
\begin{enumerate}
\item\label{it:cod1-fails} For all $(m_0, n_0) \in \natural_0 \times
\natural_0$ there exists $u \in \HH^1_{\rho^\alpha}(B^2)$ such that
\eqref{coefs-of-diff-1} fails for $(m,n) = (m_0+1, n_0)$.
\item\label{it:cod2-fails} For all $(m_0, n_0) \in \natural_0 \times
\natural_0$ there exists $u \in \HH^1_{\rho^\alpha}(B^2)$ such that
\eqref{coefs-of-diff-2} fails for $(m,n) = (m_0, n_0+1)$.
\end{enumerate}
\end{proposition}
\begin{proof} For all $(m, n) \in \natural_0 \times \natural_0$ let
\begin{equation*}
v_{m,n} := \begin{cases}
\left(2^j\right)^{2\alpha-1} (m_0+2^j+\alpha+1) (n_0+2^j+1) & \begin{minipage}{14em} if there exists $j \in \natural_0$ such that $(m,n) = (m_0 + 2^j + 1, n_0 + 2^j)$,\vspace{1ex} \end{minipage}\\
0 & \text{otherwise.}
\end{cases}
\end{equation*}
Then, on account of \eqref{squaredNorm-2d} and the asymptotic formula
\eqref{GRA}, the sum
\begin{equation*}
\sum_{(m,n) \in \natural_0 \times \natural_0} \abs{v_{m,n}}^2 \norm[n]{P\psp{\alpha}_{m,n}}_{\LL^2_{\rho^\alpha}(B^2)}^2
= \sum_{j=0}^\infty (2^j)^{4\alpha-2} (m_0 + 2^j + \alpha + 1)^2 (n_0 + 2^j + 1)^2 h\psp{\alpha}_{m_0 + 2^j + 1, n_0 + 2^j}
\end{equation*}
is finite or infinite together with $ \sum_{j=0}^\infty (2^j)^{(4\alpha-2) + 2
+ 2 - 1 - \alpha - \alpha} = \sum_{j=0}^\infty (2^j)^{2\alpha+1} $; this last
expression being, indeed, finite (as $2\alpha + 1 < 0$), it transpires that
\begin{equation*}
v := \sum_{(m,n) \in \natural_0 \times \natural_0} v_{m,n} \, P\psp{\alpha}_{m,n} \in \LL^2_{\rho^\alpha}(B^2).
\end{equation*}
The same argument goes on to show that, on defining for all $(m, n) \in
\natural_0 \times \natural_0$
\begin{equation*}
w_{m,n} := \begin{cases}
\left(2^j\right)^{2\alpha-1} (m_0+2^j+1) (n_0+2^j+\alpha+1) & \begin{minipage}{14em} if there exists $j \in \natural_0$ such that $(m,n) = (m_0 + 2^j, n_0 + 2^j + 1)$,\vspace{1ex} \end{minipage}\\
0 & \text{otherwise,}
\end{cases}
\end{equation*}
one then has
\begin{equation*}
w := \sum_{(m,n) \in \natural_0 \times \natural_0} w_{m,n} \, P\psp{\alpha}_{m,n} \in \LL^2_{\rho^\alpha}(B^2).
\end{equation*}
Using the differentiation identities in \eqref{rotated-diffs} it can be checked
that the choice of the coefficients $v_{m,n}$ and $w_{m,n}$ yields $\partial_z v
- \partial_{z^*} w = 0$ in the sense of distributions. From this and the
definition of the differential operators involved given below
\eqref{rotated-diffs} it follows that the curl of $(v + w, \imath(v - w)) \in
[\LL^2_{\rho^\alpha}(B^2)]^2 \subseteq [\LL^2(B^2)]^2$ is null in the sense of
distributions. Therefore there exists $u \in \HH^1(B^2)$, unique up to an
additive constant, such that $\grad u = (v + w, \imath(v - w))$ (cf.\
\cite[Theorem~2.9]{GR:1986}). Now, $\HH^1(B^2) \subseteq
\HH^1_{\rho^{\alpha+2}}(B^2) \subseteq \LL^2_{\rho^\alpha}(B^2)$ (see
\cite[Theorem~8.2]{Kufner:1985} for the latter inclusion, which holds with
continuous embedding). In this way we have constructed $u \in
\HH^1_{\rho^\alpha}(B^2)$ such that $\partial_{z^*} u = v$ and $\partial_z u =
w$.

On the other hand, from \eqref{with-residual} in the proof of
\autoref{lem:coefs-of-diffs} we know that \eqref{coefs-of-diff-1} holds in the
case $(m, n) = (m_0 + 1, n_0)$ if and only if
\begin{equation*}
R_{m_0 + 1, n_0, \alpha, L} = \frac{\poch{n_0+1}{L}}{\poch{n_0+\alpha+1}{L}} \frac{\poch{m_0+2}{L}}{\poch{m_0+\alpha+2}{L}} \frac{1}{m_0+n_0+2L+\alpha+2} \, v_{m_0+1+L, n_0+L} \xrightarrow{L \to \infty} 0.
\end{equation*}
However, restricting our attention to the subsequence of indices $L$ of the
form $2^j$, $j \in \natural_0$ and using the asymptotic relation \eqref{GRA},
\begin{equation*}
R_{m_0+1, n_0, \alpha, 2^j} \xrightarrow{j \to \infty} \frac{\Gamma(n_0+\alpha+1)}{\Gamma(n_0+1)} \frac{\Gamma(m_0+\alpha+2)}{\Gamma(m_0+2)} \frac{1}{2} \neq 0,
\end{equation*}
so \eqref{coefs-of-diff-1} cannot hold in this case and part
\ref{it:cod1-fails} of this proposition is proved.

Symmetry arguments analogous to those made at the end of
\autoref{lem:coefs-of-diffs} show that $(x_1, x_2) \mapsto u(x_1, -x_2)$ is a
function satisfying part \ref{it:cod2-fails}.
\end{proof}

\begin{remark} The formula analogous to \eqref{coefs-of-diff-1} and
\eqref{coefs-of-diff-1} for symmetric Jacobi expansions, namely
\begin{equation}\label{diff-coefs}
\widehat{(u')}_n\psp{\alpha} = (2k+2\alpha+1) \sum_{\substack{n=k+1\\n-k \text{ is odd}}}^\infty \frac{\poch{k+\alpha+1}{n-k}}{\poch{k+2\alpha+1}{n-k}} \hat u_{n}\psp{\alpha},
\end{equation}
where
\begin{equation*}
u = \sum_{n=0}^\infty \hat u_n\psp{\alpha} \Jac[n]{\alpha}{\alpha}
\quad\text{and}\quad
u' = \sum_{n=0}^\infty \widehat{(u')}_n\psp{\alpha} \Jac[n]{\alpha}{\alpha},
\end{equation*}
is valid for all $u \in \HH^1_{\wJac{\alpha}{\alpha}}(-1,1)$ if $\alpha
\geq -1/2$ (cf.\ \cite[eq.~2.13]{Guo:2000a}, where it is expressed in an
equivalent way in terms of Gegenbauer polynomials). Using essentially the same
arguments put forth in \autoref{lem:coefs-of-diffs} and
\autoref{pro:W12-not-enough} it can be shown that if $\alpha \in (-1, -1/2)$
then the relation \eqref{diff-coefs} is valid under the stronger condition $u
\in \HH^2_{\wJac{\alpha}{\alpha}}(-1,1)$ and that there are functions in
$\HH^1_{\wJac{\alpha}{\alpha}}(-1,1) \setminus
\HH^2_{\wJac{\alpha}{\alpha}}(-1,1)$ for which the relation is false. One
such example is the function defined by $u(x) = \int_0^x  v(t) \dd t$ where in
turn
\begin{equation*}
v = \sum_{n=0}^\infty v_n \, \Jac[n]{\alpha}{\alpha}
\quad\text{and}\quad
v_n = \begin{cases}
n^{\alpha+1} & \text{if } n \in \{2^j \mid j \in \natural_0\},\\
0 & \text{otherwise.}
\end{cases}
\end{equation*}
\end{remark}

\subsection{Main result}\label{ssc:mainResult}

Having obtained the necessary preliminary results we can prove our main result using roughly
the same outer structure of the proof of the univariate case with $\alpha = -1/2$ (Chebyshev) and
$\alpha = 0$ (Legendre) on page 302 of \cite{CHQZ-I}. The core of the argument
lies below in \autoref{lem:DPmPD} and the main result itself in
\autoref{thm:lossyProjection}. In order to express the former in a more compact
form we extend the notation $\proj\psp{\alpha}_N$ (cf.\ \eqref{projector-2d})
so that given any $k \in \natural$ and $F \in [\LL^2_{\rho^\alpha}(B^2)]^k$,
$\proj\psp{\alpha}_N(F)$ signifies the componentwise application of
$\proj\psp{\alpha}_N$ to $F$.

\begin{lemma}\label{lem:DPmPD} Let $\alpha > -1$ and $r, l \in \natural$ with
$l \geq r$. Then there exists $C = C(\alpha, l, r)$ such that for every $N \in
\natural$ and $u \in \HH^l_{\rho^\alpha}(B^2)$,
\begin{equation*}
\norm{\proj\psp{\alpha}_N(\grad_r u) - \grad_r \proj\psp{\alpha}_N(u)}_{[\LL^2_{\rho^\alpha}(B^2)]^{r+1}}
\leq C N^{2r - 1/2 - l} \norm{u}_{\HH^l_{\rho^\alpha}(B^2)}.
\end{equation*}
\end{lemma}
\begin{proof} Let $l \in \natural$ and $u \in \HH^l_{\rho^\alpha}(B^2)$. If
we prove the existence of $C>0$ independent of $u$ and $N$ such that
\begin{equation}\label{lossyGoal}
\norm{\proj\psp{\alpha}_N\left(\partial_{z^*} u\right) - \partial_{z^*} \proj\psp{\alpha}_N(u)}_{\LL^2_{\rho^\alpha}(B^2)}^2
\leq C N^{3 - 2 l} \norm{u}_{\HH^l_{\rho^\alpha}(B^2)}^2
\end{equation}
and the corresponding result involving the operator $\partial_z$, the $r = 1$
case of our desired result will follow. For the proof of \eqref{lossyGoal} we
can assume that $u$ is regular enough for the relation \eqref{coefs-of-diff-1}
between the orthogonal expansion of $u$ and that of its image under the
operator $\partial_{z^*}$ to hold (otherwise, we can replace $u$ by the members
of a sequence of $\CC^\infty(\overline{B^2})$ functions which converges to $u$
in $\HH^l_{\rho^\alpha}(B^2)$---which exists by virtue of
\cite[Remark~11.12.(iii)]{Kufner:1985}---and once \eqref{lossyGoal} is proved
it will extend to $u$ by continuity); that is,
\begin{equation*}
u = \sum_{m=0}^\infty \sum_{n=0}^\infty \hat u\psp{\alpha}_{m,n} \, P\psp{\alpha}_{m,n}
\quad\text{and}\quad \partial_{z^*} u = \sum_{m=0}^\infty \sum_{n=0}^\infty v_{m,n} \, P\psp{\alpha}_{m,n},
\end{equation*}
both series converging in the $\LL^2_{\rho^\alpha}(B^2)$ sense, with the
$v_{m,n}$ and the $\hat u\psp{\alpha}_{m,n}$ connected by
\eqref{coefs-of-diff-1}. As $\proj\psp{\alpha}_N(u)$ is a polynomial, it is
also regular enough to have the coefficients of its Fourier--Zernike series and
the corresponding coefficients of its image under the operator $\partial_{z^*}$
connected by the formula \eqref{coefs-of-diff-1}. Further taking into account
the fact that the expansion of $\proj\psp{\alpha}_N(u)$ is but a truncation of
the expansion of $u$ we have
\begin{equation*}
\proj\psp{\alpha}_N(u) = \sum_{m+n \leq N} \hat u\psp{\alpha}_{m,n} \, P\psp{\alpha}_{m,n}
\quad\text{and}\quad
\partial_{z^*} \proj\psp{\alpha}_N(u) = \sum_{m+n \leq N} v\psp{\mathrm{trunc}}_{m,n} \, P\psp{\alpha}_{m,n},
\end{equation*}
where \eqref{coefs-of-diff-1} takes the particular form: for all $(m, n) \in
\natural_0 \times \natural_0$ with $m+n \leq N$,
\begin{equation*}
v\psp{\mathrm{trunc}}_{m,n} = (m + n + \alpha + 1) \sum_{l = 0}^{\floor{\frac{N-m-n-1}{2}}} \frac{\poch{m+1}{l}}{\poch{m+\alpha+1}{l}} \frac{\poch{n+1}{l+1}}{\poch{n+\alpha+1}{l+1}} \, u\psp{\alpha}_{m+l, n+1+l}.
\end{equation*}
In particular, $v\psp{\mathrm{trunc}}_{m,n} = 0$ if $m + n = N$. Therefore,
whenever $0 \leq m+n \leq N$ and adopting the notation $\delta\psp{N}_{m,n} =
\floor{\frac{N-m-n+1}{2}}$,
\begin{multline}\label{identified-diff-coef}
\frac{v_{m,n} - v\psp{\mathrm{trunc}}_{m,n}}{m + n + \alpha + 1}
= \sum_{l = \delta\psp{N}_{m,n}}^\infty \frac{\poch{m+1}{l}}{\poch{m+\alpha+1}{l}} \frac{\poch{n+1}{l+1}}{\poch{n+\alpha+1}{l+1}} \, u\psp{\alpha}_{m+l, n+1+l}\\
= \sum_{l = 0}^\infty \frac{\poch{m+1}{l+\delta\psp{N}_{m,n}}}{\poch{m+\alpha+1}{l+\delta\psp{N}_{m,n}}} \frac{\poch{n+1}{l+\delta\psp{N}_{m,n}+1}}{\poch{n+\alpha+1}{l+\delta\psp{N}_{m,n}+1}} \, u\psp{\alpha}_{m+l+\delta\psp{N}_{m,n}, n+1+l+\delta\psp{N}_{m,n}}\\
= \frac{\poch{m+1}{\delta\psp{N}_{m,n}}}{\poch{m+\alpha+1}{\delta\psp{N}_{m,n}}} \frac{\poch{n+1}{\delta\psp{N}_{m,n}}}{\poch{n+\alpha+1}{\delta\psp{N}_{m,n}}} \frac{v_{m+\delta\psp{N}_{m,n},n+\delta\psp{N}_{m,n}}}{m + n + \alpha + 2 \delta\psp{N}_{m,n} + 1},
\end{multline}
where the last equality is obtained by expanding the Pochhammer symbols of the
form $\poch{X}{\delta\psp{N}_{m,n}+Y}$ according to the rules given in
\autoref{ssc:notation} and noting that then \eqref{coefs-of-diff-1} can be used
to make the coefficient $v_{m+\delta\psp{N}_{m,n}, n+\delta\psp{N}_{m,n}}$
appear. Now, from \eqref{squaredNorm-2d}
\begin{equation}\label{two-ratios}
h\psp{\alpha}_{m,n} =
\frac{m+n+\alpha+2\delta\psp{N}_{m,n}+1}{m+n+\alpha+1} \frac{\poch{m+\alpha+1}{\delta\psp{N}_{m,n}} \poch{n+\alpha+1}{\delta\psp{N}_{m,n}}}{\poch{m+1}{\delta\psp{N}_{m,n}} \poch{n+1}{\delta\psp{N}_{m,n}}} \, h\psp{\alpha}_{m+\delta\psp{N}_{m,n}, n+\delta\psp{N}_{m,n}}.
\end{equation}
Using \eqref{identified-diff-coef} and the fact that if $m+n = N$ then
$\delta\psp{N}_{m,n} = 0$,
\begin{multline}\label{pdmdp-1-even}
\proj\psp{\alpha}_N\left(\partial_{z^*} u\right) - \partial_{z^*} \proj\psp{\alpha}_N(u)
= \sum_{m+n \leq N-1} \left(v_{m,n} - v\psp{\mathrm{trunc}}_{m,n}\right) P\psp{\alpha}_{m,n} + \sum_{m+n = N} v_{m,n} \, P\psp{\alpha}_{m,n}\\
= \sum_{k=0}^N \sum_{m+n=k} (k+\alpha+1) \frac{\poch{m+1}{\delta\psp{N}_{m,n}}}{\poch{m+\alpha+1}{\delta\psp{N}_{m,n}}} \frac{\poch{n+1}{\delta\psp{N}_{m,n}}}{\poch{n+\alpha+1}{\delta\psp{N}_{m,n}}} \frac{v_{m+\delta\psp{N}_{m,n},n+\delta\psp{N}_{m,n}}}{k+\alpha+2\delta\psp{N}_{m,n}+1} P\psp{\alpha}_{m,n}.
\end{multline}
As the terms resulting in \eqref{pdmdp-1-even} are
$\LL^2_{\rho^\alpha}(B^2)$-orthogonal to each other, taking the corresponding
squared norm of both its ends, using \eqref{two-ratios} results in
\begin{multline*}
\norm{\proj\psp{\alpha}_N\left(\partial_{z^*} u\right) - \partial_{z^*} \proj\psp{\alpha}_N(u)}_{\LL^2_{\rho^\alpha}(B^2)}^2\\
= \sum_{k=0}^N \sum_{m+n=k} \frac{k + \alpha + 1}{k + \alpha + 2\delta\psp{N}_{m,n} + 1} \frac{\poch{m + 1}{\delta\psp{N}_{m,n}}}{\poch{m + \alpha + 1}{\delta\psp{N}_{m,n}}} \frac{\poch{n + 1}{\delta\psp{N}_{m,n}}}{\poch{n + \alpha + 1}{\delta\psp{N}_{m,n}}} \abs{v_{m+\delta\psp{N}_{m,n}, n+\delta\psp{N}_{m,n}}}^2 h\psp{\alpha}_{m+\delta\psp{N}_{m,n}, n+\delta\psp{N}_{m,n}}.
\end{multline*}
We want to rearrange the above sum so that those $(m',n') \in \natural_0 \times
\natural_0$ such that $\abs{v_{m',n'}}^2 h\psp{\alpha}_{m',n'}$ appears in the
above sum and their accompanying coefficients become readily apparent. Let
$\mathcal{E}\psp{\alpha}_N$ and $\mathcal{O}\psp{\alpha}_N$ denote the above
sum restricted to the terms with $N-k$ even and odd, respectively. In the inner
sum of both resulting expressions $n$ can be replaced with $k-m$ by letting $m$
range in $\{0, \dotsc, k\}$. Applying the change of variable $(i, j) =
(m+\frac{N-k}{2}, \frac{N-k}{2})$ in the sum defining
$\mathcal{E}\psp{\alpha}_N$ we are left with
\begin{subequations}\label{evenPart}
\begin{equation}\label{evenPartA}
\mathcal{E}\psp{\alpha}_N = \sum_{i=0}^N \left[ \sum_{j=0}^{\min(i,N-i)} \mathcal{E}\psp{\alpha}_{N,i,j} \right] \abs{v_{i, N-i}}^2 h\psp{\alpha}_{i, N-i}.
\end{equation}
where
\begin{equation}\label{evenPartB}
\mathcal{E}\psp{\alpha}_{N,i,j} := \frac{N - 2j + \alpha + 1}{N + \alpha + 1} \frac{\poch{i - j + 1}{j}}{\poch{i - j + \alpha + 1}{j}} \frac{\poch{N - i - j + 1}{j}}{\poch{N - i - j + \alpha + 1}{j}}.
\end{equation}
\end{subequations}
%
%
The sum inside the square brackets in \eqref{evenPartA} is invariant under the
transformation $i \mapsto N-i$. Thus, we can learn the values of all the
instances of this sum by looking at the cases where $i \leq N-i$ only. For such
$i$ it is straightforward to check that as long as $j \in \{0, \dotsc, i\}$,
\begin{equation*}
\mathcal{E}\psp{\alpha}_{N,i,j}
= \begin{cases}
\displaystyle \Delta_j \left[ -\frac{(i-j+\alpha+1) (N-i-j+\alpha+1) \poch{i-j+1}{j} \poch{N-i-j+1}{j}}{(\alpha+1) (N+\alpha+1) \poch{i-j+\alpha+1}{j} \poch{N-i-j+\alpha+1}{j}} \right] & \text{if } \alpha \neq 0,\\
\displaystyle \frac{N - 2j + \alpha + 1}{N + \alpha + 1} & \text{if } \alpha = 0.
\end{cases}
\end{equation*}
Hence, the sum with respect to $j$ telescopes if $\alpha \neq 0$ and is well known if
$\alpha = 0$, giving (using the abovementioned invariance under the
transformation $i \mapsto N-i$)
\begin{equation*}
\mathcal{E}\psp{\alpha}_N = \sum_{i=0}^N \frac{(i+\alpha+1) (N-i+\alpha+1)}{(\alpha+1) (N+\alpha+1)} \abs{v_{i,N-i}}^2 h\psp{\alpha}_{i,N-i}.
\end{equation*}
Applying the change of variable $(i, j) = (m+\frac{N-k+1}{2}, \frac{N-k+1}{2})$
in the sum defining $\mathcal{O}\psp{\alpha}_N$ we obtain
\begin{subequations}\label{oddPart}
\begin{equation}\label{oddPartA}
\mathcal{O}\psp{\alpha}_N = \sum_{i=1}^N \left[\sum_{j=1}^{\min(i,N+1-i)} \mathcal{O}\psp{\alpha}_{N,i,j} \right] \abs{v_{i,N+1-i}}^2 h\psp{\alpha}_{i,N+1-i}
\end{equation}
where
\begin{equation}\label{oddPartB}
\mathcal{O}\psp{\alpha}_{N,i,j} := \frac{N-2j+\alpha+2}{N+\alpha+2} \frac{\poch{i-j+1}{j}}{\poch{i-j+\alpha+1}{j}} \frac{\poch{N-i-j+2}{j}}{\poch{N-i-j+\alpha+2}{j}}.
\end{equation}
\end{subequations}
%
The sum inside the square brackets in \eqref{oddPartA} is invariant under the
transformation $i \mapsto N+1-i$. Also, comparing \eqref{oddPartB} with
\eqref{oddPartA} we find that $\mathcal{O}\psp{\alpha}_{N,i,j} =
\mathcal{E}\psp{\alpha}_{N+1,i,j}$. Hence, we can adapt our previous argument
and state
\begin{equation*}
\mathcal{O}\psp{\alpha}_N = \sum_{i=1}^N \frac{i(N+1-i)}{(\alpha+1)(N+\alpha+2)} \abs{v_{i,N+1-i}}^2 h\psp{\alpha}_{i,N+1-i}.
\end{equation*}
Summing the resulting expressions for $\mathcal{E}\psp{\alpha}_N$ and
$\mathcal{O}\psp{\alpha}_N$ and using the fact that $i \mapsto
(i+\alpha+1)(N-i+\alpha+1)$ and $i \mapsto i(N+1-i)$, seen as functions of a
real variable, attain their maxima at $N/2$ and $(N+1)/2$, respectively, we
obtain
\begin{multline*}
\norm{\proj\psp{\alpha}_N\left(\partial_{z^*} u\right) - \partial_{z^*} \proj\psp{\alpha}_N(u)}_{\LL^2_{\rho^\alpha}(B^2)}^2\\
\leq \frac{(N/2+\alpha+1)^2}{(\alpha+1) (N+\alpha+1)} \sum_{m+n=N} \abs{v_{m,n}}^2 h\psp{\alpha}_{m,n} + \frac{((N+1)/2)^2}{(\alpha+1)(N+\alpha+2)} \sum_{m+n=N+1} \abs{v_{m,n}}^2 h\psp{\alpha}_{m,n}\\
\leq C_{\alpha} (N+1) \left( \norm{\partial_{z^*} u - \proj\psp{\alpha}_{N-1}\left(\partial_{z^*} u\right)}_{\LL^2_{\rho^\alpha}(B^2)}^2 + \norm{\partial_{z^*} u - \proj\psp{\alpha}_N\left(\partial_{z^*} u\right)}_{\LL^2_{\rho^\alpha}(B^2)}^2 \right)\\
\leq C_\alpha C (N+1) N^{2(1-l)} \norm{\partial_{z^*} u}_{\HH^{l-1}_{\rho^\alpha}(B^2)}^2 + C_\alpha C (N+1) (N+1)^{2(1-l)} \norm{\partial_{z^*} u}_{\HH^{l-1}_{\rho^\alpha}(B^2)}^2,
\end{multline*}
where $C_\alpha = \sup_{N \in \natural_0} \max\left(
\frac{(N/2+\alpha+1)^2}{(\alpha+1) (N+\alpha+1) (N+1)},
\frac{((N+1)/2)^2}{(\alpha+1)(N+\alpha+2)(N+1)} \right) > 0$ and the last
inequality comes from \autoref{cor:L2-approx}. Upon using standard inequalities
\eqref{lossyGoal} is attained.

By using the reflection introduced at the end of the proof of
\autoref{lem:coefs-of-diffs} we can turn \eqref{lossyGoal} into its analogue
for the $\partial_z$ differential operator and thus conclude the proof of the
$r = 1$ case of this lemma.

Starting from the bidimensional case of the Markov inequality in
\autoref{lem:Markov} it is readily proved by induction that there exists $C =
C(\alpha,r) > 0$ such that for all $N \in \natural_0$ and $p \in \poly^d_N$,
\begin{equation}\label{iteratedMarkov}
\abs{p}_{\HH^r_{\rho^\alpha}(B^2)} \leq C N^{2r} \norm{p}_{\LL^2_{\rho^\alpha}(B^2)}.
\end{equation}
We are now in a position to obtain the general case of this lemma by induction
on $r$, the initialization $r = 1$ having already being proved. Thus, let us
assume that the desired result holds up to some $r \in \natural$ and let $l
\geq r+1$ and $u \in \HH^l_{\rho^\alpha}(B^2)$. Then,
\begin{multline*}
\norm{\proj\psp{\alpha}_N\left(\grad_r \partial_{z^*} u\right) - \grad_r \partial_{z^*} \proj\psp{\alpha}_N(u)}_{[\LL^2_{\rho^\alpha}(B^2)]^{r+1}}\\
\leq \norm{\proj\psp{\alpha}_N\left(\grad_r \partial_{z^*} u\right) - \grad_r \proj\psp{\alpha}_N\left(\partial_{z^*} u\right)}_{[\LL^2_{\rho^\alpha}(B^2)]^{r+1}} + \abs{\proj\psp{\alpha}_N\left(\partial_{z^*} u\right) - \partial_{z^*} u \proj\psp{\alpha}_N(u)}_{\HH^r_{\rho^\alpha}(B^2)}\\
\leq C N^{-1/2+2r-(l-1)} \norm{\partial_{z^*} u}_{\HH^{l-1}_{\rho^\alpha}(B^2)} + C N^{2r} \norm{\proj\psp{\alpha}_N\left(\partial_{z^*} u\right) - \partial_{z^*} \proj\psp{\alpha}_N(u)}_{\LL^2_{\rho^\alpha}(B^2)},
\end{multline*}
where we have used the induction hypothesis and \eqref{iteratedMarkov}. Using
\eqref{lossyGoal} to bound the last term above we obtain
\begin{equation*}
\norm{\proj\psp{\alpha}_N\left(\grad_r \partial_{z^*} u\right) - \grad_r \partial_{z^*} \proj\psp{\alpha}_N(u)}_{[\LL^2_{\rho^\alpha}(B^2)]^{r+1}} \leq C N^{-1/2 + 2(r+1) - l} \norm{u}_{\HH^l_{\rho^\alpha}(B^2)}.
\end{equation*}
Combining this with its analogue involving the $\partial_z$ operator we obtain
the desired bound for the commutator of the projection and the $\grad_{r+1}$
operators.
\end{proof}

\begin{remark}\label{rem:easierCase}
The proof of \autoref{lem:DPmPD} can be significantly simplified in the $\alpha \geq 0$ case because then $\mathcal{E}\psp{\alpha}_{N,i,j}$ of \eqref{evenPartB} and $\mathcal{O}\psp{\alpha}_{N,i,j}$ of \eqref{oddPartB} can each be bounded by $1$.
\end{remark}

\begin{thm}\label{thm:lossyProjection} Let $\alpha > -1$ and $r, l \in
\natural$ with $l \geq r$. Then there exists $C = C(\alpha,l,r) > 0$ such that
for every $N \in \natural$ and $u \in \HH^l_{\rho^\alpha}(B^2)$,
\begin{equation*}
\norm[n]{u - \proj\psp{\alpha}_N(u)}_{\HH^r_{\rho^\alpha}(B^2)} \leq C \, N^{-1/2+2r-l} \norm{u}_{\HH^l_{\rho^\alpha}(B^2)}.
\end{equation*}
\end{thm}
\begin{proof} For every $k \in \{1, \dotsc, r\}$,
\begin{multline*}
\norm[n]{\grad_k \big(u - \proj\psp{\alpha}_N(u)\big)}_{[\LL^2_{\rho^\alpha}(B^2)]^{k+1}}^2\\
\leq 2\norm[n]{\grad_k u - \proj\psp{\alpha}_N(\grad_k u)}_{[\LL^2_{\rho^\alpha}(B^2)]^{k+1}}^2 + 2\norm[n]{\proj\psp{\alpha}_N(\grad_k u) - \grad_k \proj\psp{\alpha}_N(u)}_{[\LL^2_{\rho^\alpha}(B^2)]^{k+1}}^2.
\end{multline*}
We can bound the first term using \autoref{cor:L2-approx} and the second term
using \autoref{lem:DPmPD}. As the squared $\HH^l_{\rho^\alpha}(B^2)$ norm
of $u - \proj\psp{\alpha}_N(u)$ is the sum of the squared
$\LL^2_{\rho^\alpha}(B^2)$ norm of $u - \proj\psp{\alpha}_N(u)$ (which again,
can be bounded using \autoref{cor:L2-approx}) and the left-hand side above for
$k \in \{1, \dotsc, r\}$, we obtain the desired bound upon realizing that the
highest power on $N$ which will appear is $-1 - 2l + 4 r$ and taking square
roots.
\end{proof}

Given $l \in \natural_0$ and $\theta \in (0,1)$ we use complex interpolation
(see \cite[\P7.51--52]{AF:2003} for a succinct discussion which suffices for our
purposes save for a strong enough statement of the reiteration theorem, which
can be found in \cite[\P12.3]{Calderon:1964}) to define
\begin{equation}\label{interpolatedSpace}
\HH^{l+\theta}_{\rho^\alpha}(B^2) := \left[\HH^l_{\rho^\alpha}(B^2), \HH^{l+1}_{\rho^\alpha}(B^2)\right]_\theta.
\end{equation}
Then, by using the exact interpolation theorem, \autoref{cor:L2-approx} and
\autoref{thm:lossyProjection} are readily generalized to the above intermediate
spaces:

\begin{corollary}\label{cor:interp} Let $\alpha > -1$ and $r, l \geq 0$ with $l
\geq r$. Then there exists $C = C(\alpha,l,r) > 0$ such that for every $N \in
\natural$ and $u \in \HH^l_{\rho^\alpha}(B^2)$,
\begin{subequations}\label{interp}
\begin{equation}
\norm[n]{u - \proj\psp{\alpha}_N(u)}_{\HH^r_{\rho^\alpha}(B^2)} \leq C N^{e(l,r)} \norm{u}_{\HH^l_{\rho^\alpha}(B^2)}
\end{equation}
where
\begin{equation}\label{growthRate}
e(l,r) := \begin{cases}
3/2 \, r - l & \text{if } 0 \leq r \leq 1,\\
-1/2 + 2\,r - l & \text{if } r \geq 1. \end{cases}
\end{equation}
\end{subequations}
\end{corollary}
\begin{proof} For $N \in \natural$ let $T\psp{\alpha}_{N,l,r}$ denote the operator $I-\proj\psp{\alpha}_N
\colon \HH^l_{\rho^\alpha}(B^2) \to \HH^r_{\rho^\alpha}(B^2)$. Let us suppose first that neither $l$ nor $r$ is an
integer. Then, if $\floor{l} \geq \floor{r}+1$, for $j \in \{0,1\}$, using the known
bounds on the operator norms
$\norm[n]{T\psp{\alpha}_{N,\floor{l},\floor{r}+j}}$ and
$\norm[n]{T\psp{\alpha}_{N,\floor{l}+1,\floor{r}+j}}$ and the exact
interpolation theorem with interpolation parameter $l - \floor{l}$ results in
the desired bound on $\norm[n]{T\psp{\alpha}_{N,l,\floor{r}}}$ and
$\norm[n]{T\psp{\alpha}_{N,l,\floor{r}+1}}$; combining these estimates with the
exact interpolation theorem with interpolation parameter $r - \floor{r}$ gives
the desired bound on $\norm[n]{T\psp{\alpha}_{N,l,r}}$. If $\floor{l} = \floor{r}$,
the bound on $\norm[n]{T\psp{\alpha}_{N,l,\floor{r}}}$ is obtained exactly as
in the previous case but now it is combined via the exact interpolation theorem
with parameter $\theta = \frac{r-\floor{r}}{l-\floor{l}}$ with the desired bound on
$\norm[n]{T\psp{\alpha}_{N,l,l}}$ which, in turn, comes about by combining the
known bounds on $\norm[n]{T\psp{\alpha}_{N,\floor{l},\floor{l}}}$ and
$\norm[n]{T\psp{\alpha}_{N,\floor{l}+1,\floor{l}+1}}$ with the exact
interpolation theorem with parameter $l-\floor{l}$; the reiteration theorem
is used to ensure that
\begin{multline*}
\left[\HH^{\floor{r}}_{\rho^\alpha}(B^2),\HH^l_{\rho^\alpha}(B^2)\right]_\theta
= \left[\left[\HH^{\floor{l}}_{\rho^\alpha}(B^2),\HH^{\floor{l}+1}_{\rho^\alpha}(B^2)\right]_0,\left[\HH^{\floor{l}}_{\rho^\alpha}(B^2),\HH^{\floor{l}+1}_{\rho^\alpha}(B^2)\right]_{l-\floor{l}}\right]_\theta\\
= \left[\HH^{\floor{l}}_{\rho^\alpha}(B^2),\HH^{\floor{l}+1}_{\rho^\alpha}(B^2)\right]_{(1-\theta) 0 + \theta (l-\floor{l})}
= \HH^r_{\rho^\alpha}(B^2).
\end{multline*}
The cases where either $l$ or $r$ is an integer are similar but simpler so we omit further details.
\end{proof}

\begin{remark}\label{rem:interp}\hfill
\begin{enumerate}
\item Essentially the same argument put forward in \autoref{cor:interp} allows
for generalizing \autoref{cor:L2-approx} and \autoref{thm:lossyProjection}
using real instead of complex interpolation.
\item The proof of \autoref{cor:interp} works with the interpolated space
defined as in \eqref{interpolatedSpace} and does not depend on any further
characterization of those spaces (cf.\ \cite[Lemma~2.1]{CQ:1982}, where a weighted
identity of the form $\HH^{\theta m} = [\LL^2,\HH^m]_\theta$, for $(m,\theta)
\in \natural \times (0,1)$, is tacitly used).
\end{enumerate}
\end{remark}

\subsection{On the sharpness of the main result}\label{ssc:sharpness}

\subsubsection*{Proved sharpness results}

We can show the optimality with respect to the power on $N$ of
\autoref{thm:lossyProjection} in the $r = 1$ case and that of its $r = 0$
analogue, namely \autoref{cor:L2-approx}, in the two-dimensional case. Note that
in both cases all the weighted Sobolev spaces involved have integer regularity
parameters. We will need the following auxiliary result, which is of
independent interest (see Proposition~4.26 and Theorem~4.29 of \cite{Figueroa}
for its one-dimensional analogue and one application, respectively).
\begin{proposition}\label{pro:wrong-seminorm} For all $\alpha > -1$ and $(m, n)
\in \natural_0 \times \natural_0$,
\begin{equation*}
\abs{P\psp{\alpha}_{m,n}}_{\HH^1_{\rho^\alpha}(B^2)}^2 = \frac{2 \, \pi \, \Gamma(\alpha+1)^2 \Gamma(m+1) \, \Gamma(n+1)}{(\alpha+1) \Gamma(m+\alpha+1) \, \Gamma(n+\alpha+1)} (2 \, m \, n + (m+n)(\alpha+1)).
\end{equation*}
\end{proposition}
\begin{proof} We first observe that for all $\alpha > -1$ and $(m, n) \in
\natural_0 \times \natural_0$,
\begin{equation}\label{wrongNorm}
\norm{P\psp{\alpha+1}_{m,n}}_{\LL^2_{\rho^\alpha}(B^2)}^2
= \frac{\pi \, \Gamma(m+1) \, \Gamma(n+1) \, \Gamma(\alpha+1) \, \Gamma(\alpha+2)}{\Gamma(m+\alpha+2) \, \Gamma(n+\alpha+2)}.
\end{equation}
Indeed, using \autoref{pro:Zernike-CC} to expand $P\psp{\alpha+1}_{m,n}$
in terms of the $P\psp{\alpha}_{m-k, n-k}$ with $k \in \{0, \dotsc,
\min(m,n)\}$, noting that the latter are $\LL^2_{\rho^\alpha}(B^2)$-orthogonal
and using \eqref{squaredNorm-2d} we obtain
\begin{equation*}
\norm{P\psp{\alpha+1}_{m,n}}_{\LL^2_{\rho^\alpha}(B^2)}^2
= \frac{\pi \, \Gamma(m+1)^2 \, \Gamma(n+1)^2 \, \Gamma(\alpha+2)^2}{\Gamma(\alpha+m+2)^2 \, \Gamma(\alpha+n+2)^2} \sum_{k=0}^{\min(m,n)} \theta\psp{\alpha}_{m,n,k}
\end{equation*}
where $\theta\psp{\alpha}_{m,n,k} = (\alpha+m+n-2k+1)
\frac{\Gamma(\alpha+m-k+1) \Gamma(\alpha+n-k+1)}{\Gamma(m-k+1) \Gamma(n-k+1)}$.
Using that $\theta\psp{\alpha}_{m,n,k} = \Delta_k(\zeta\psp{\alpha}_{m,n,k})$,
where $\zeta\psp{\alpha}_{m,n,k} = -\frac{\Gamma(m-k+\alpha+2)
\Gamma(n-k+\alpha+2)}{(\alpha+1) \Gamma(m-k+1) \Gamma(n-k+1)}$ the above sum
telescopes and the \eqref{wrongNorm} follows. Then the desired result is a
direct consequence of the relations in \eqref{rotated-diffs} and
\eqref{wrongNorm}.
\end{proof}

\begin{thm}\label{thm:optimality} The power on $N$ for the $d = 2$ case of
\autoref{cor:L2-approx} and that of the $r = 1$ case of
\autoref{thm:lossyProjection} when $r = 1$ are sharp.
\end{thm}
\begin{proof}
Let $\alpha > -1$. By iterating the relations in \eqref{rotated-diffs} it is
readily proved by induction that for every $m, n, l_1, l_2 \in \natural_0$,
\begin{equation}\label{higher-rotated-diffs}
\partial_{z^*}^{l_2} \partial_z^{l_1} P\psp{\alpha}_{m,n} = \frac{\poch{m-l_1+1}{l_1} \poch{n-l_2+1}{l_2} \poch{n+\alpha+1}{l_1} \poch{m+\alpha+1}{l_2}}{\poch{\alpha+1}{l_1+l_2}} P\psp{\alpha+l_1+l_2}_{m-l_1,n-l_2}.
\end{equation}
Given $l, j \in \natural$ with $j \geq l$ we define the polynomial
\begin{equation}\label{t}
t\psp{\alpha,l}_j := \sum\limits_{k=0}^l \frac{\poch{-l}{k} \Gamma(\alpha+j+l-k+1)^2 (\alpha+2j+2l-2k+1)}{\Gamma(k+1) \Gamma(j+l-k+1)^2 \poch{\alpha+2j+l-k+1}{l+1}} P\psp{\alpha}_{j+l-k,j+l-k}.
\end{equation}
Let $l_1, l_2 \in \natural_0$ be such that $l_1 + l_2 = l$. Then applying the
the $\partial_{z^*}^{l_2} \partial_z^{l_1}$ operator
to $t\psp{\alpha,l}_j$ and using \eqref{higher-rotated-diffs} results in a
linear combination of Zernike polynomials of parameter $\alpha + l$. By
comparing the resulting expression term by term with the result of applying
\autoref{pro:Zernike-CC} with $(\alpha,\gamma,m,n)$ replaced by
$(\alpha,\alpha+l,j+l_2,j+l_1)$ (which, because $\poch{-l}{k} = 0$ if $k \geq
l+1$, is a sum with the same number of effective terms) we find that
\begin{equation*}
\partial_{z^*}^{l_2} \partial_z^{l_1} t\psp{\alpha,l}_j
= \frac{\Gamma(\alpha+j+l_1+1) \Gamma(\alpha+j+l_2+1)}{\Gamma(j+l_1+1) \Gamma(j+l_2+1)} P\psp{\alpha}_{j+l_2,j+l_1},
\end{equation*}
whence, using \eqref{GRA}, \eqref{squaredNorm-2d}, and \eqref{proto-equivalent-seminorm},
\begin{multline}\label{t-l-seminorm}
\abs{t\psp{\alpha,l}_j}_{\HH^l_{\rho^\alpha}(B^2)}^2
\cong_l \sum_{q = 0}^l \norm{\partial_{z^*}^q \partial_z^{l-q} t\psp{\alpha,l}_j}_{\LL^2_{\rho^\alpha}(B^2)}^2\\
= \frac{\pi \, \Gamma(\alpha+1)^2}{2j+l+\alpha+1} \sum_{q = 0}^l \frac{\Gamma(\alpha+j+l-q+1) \Gamma(\alpha+j+q+1)}{\Gamma(j+l-q+1) \Gamma(j+q+1)}.
\end{multline}
Let us define for integer $j \geq l$ the indices $N\psp{l}_j$ and the residuals
$R\psp{\alpha,l}_j$ by
\begin{equation}\label{indexAndResidual}
N\psp{l}_j := 2j + 2l -1
\qquad\text{and}\qquad
R\psp{\alpha,l}_j :=  t\psp{\alpha,l}_j - \proj\psp{\alpha}_{N\psp{l}_j}\left(t\psp{\alpha,l}_j\right).
\end{equation}
As $R\psp{\alpha,l}_j$ is exactly the $k = 0$ term of the sum in \eqref{t},
using \eqref{squaredNorm-2d},
\begin{equation}\label{R-L2-norm}
\norm{R\psp{\alpha,l}_j}_{\LL^2_{\rho^\alpha}(B^2)}^2  = \frac{\pi \Gamma(\alpha+1)^2}{2j+2l+\alpha+1} \frac{\Gamma(\alpha+j+l+1)^2}{\Gamma(j+l+1)^2} \frac{\Gamma(\alpha+2j+l+1)^2}{\Gamma(\alpha+2j+2l+1)^2}
\end{equation}
and, by \autoref{pro:wrong-seminorm},
\begin{equation}\label{R-1-seminorm}
\abs{R\psp{\alpha,l}_j}_{\HH^1_{\rho^\alpha}(B^2)}^2
= \frac{4 \, \pi \, \Gamma(\alpha+1)^2 \Gamma(\alpha+j+l+1)^2 (\alpha+2j+2l+1)^2}{(\alpha+1) \Gamma(j+l+1)^2 \poch{\alpha+2j+l+1}{l+1}^2} (j+l) (\alpha+j+l+1).
\end{equation}
%
Thus, using the asymptotic formula \eqref{GRA}, for $r \in \{0,1\}$,
\begin{equation}\label{knownRate}
\left. \abs{R\psp{\alpha,l}_j}_{\HH^r_{\rho^\alpha}(B^2)} \right/ \abs{t\psp{\alpha,l}_j}_{\HH^l_{\rho^\alpha}(B^2)}
\sim C_{\alpha,l,r} j^{3/2 r - l}
\sim \tilde C_{\alpha,l,r} (N\psp{l}_j)^{3/2 r - l}
\end{equation}
as $j \to \infty$. By the norm equivalence of \autoref{pro:equivalentNorm} and
the fact that the Fourier--Zernike series of $t\psp{\alpha,l}_j$ only have
non-zero terms of degrees between $2j$ and $2j+2l$ we can replace the seminorms
above by the corresponding norms. Thus, the choice $(u, N) =
(t\psp{\alpha,l}_j, N\psp{l}_j)$ turns the inequalities of \autoref{cor:L2-approx} and that of the $r = 1$ case of
\autoref{thm:lossyProjection} into asymptotic equalities; hence, the power on
$N$ in each case is sharp.
\end{proof}

\subsubsection*{Conjectured sharpness of main result in general}

We conjecture that \autoref{thm:lossyProjection} is optimal with respect to the
power on $N$ in the $r \geq 2$ case as well and that, for every $\alpha > -1$
and $l \in \natural$, the same sequence $(t\psp{\alpha,l}_j)_{j=l}^\infty$
defined in \eqref{t} attains the proved upper-bound rate
asymptotically in the same way it does so in the proof of
\autoref{thm:optimality}; that is, we conjecture that for all $\alpha > -1$, $l
\in \natural$ and $r \in \{1, \dotsc, l\}$, as $j \to \infty$,
\begin{equation}\label{conjecturedRate}
\rat\psp{\alpha,l}_{r,j} := \left. \abs{R\psp{\alpha,l}_j}_{\HH^r_{\rho^\alpha}(B^2)} \right/ \abs{t\psp{\alpha,l}_j}_{\HH^l_{\rho^\alpha}(B^2)}
\sim C_{\alpha,l,r} j^{-1/2 + 2 r - l}
\sim \tilde C_{\alpha,l,r} (N\psp{l}_j)^{-1/2 + 2 r - l}
\end{equation}
with $N\psp{l}_j$ and $R\psp{\alpha,l}_j$ as in \eqref{indexAndResidual} and
with the same immateriality of using seminorms instead of norms discussed in the
proof of \autoref{thm:optimality}.

\subsubsection*{Numerical tests in the proved and conjectured cases} The
conjecture \eqref{conjecturedRate} is informed by numerical tests coded in the
Julia programming language\footnote{Code available from
\url{https://github.com/lfiguero/ZernikeSuite}.}.
The code has the encoding of a polynomial by its finite Fourier--Zernike coefficients as its basic data structure and is able to differentiate, perform some changes of basis and compute inner products and norms of the represented polynomials mainly by using equations \eqref{squaredNorm-2d}, \eqref{rotated-diffs} and those of \autoref{pro:extra-expansions}.
We numerically compute the
ratio in the leftmost expression of \eqref{conjecturedRate} including both the
proved $r = 0$ and $r = 1$ cases and the conjectured $r \geq 2$ case. The
behavior of a representative instance of those tests is shown in
\autoref{fig:cst} and \autoref{tab:cst}; there $\rat\psp{\alpha,l}_{r,j}$ is
the seminorm ratio defined in \eqref{conjecturedRate} and the experimental
growth rate with respect to the truncation degree is
\begin{equation}
\experimentalRate_r := \frac{\log\left(\left.\rat\psp{\alpha,l}_{r,j}\right/\rat\psp{\alpha,l}_{r,j'}\right)}{\log\left(\left.N\psp{l}_j\right/N\psp{l}_{j'}\right)},
\end{equation}
where $N\psp{l}_j$ and $N\psp{l}_{j'}$ are consecutive truncation degrees in
the table (in one-to-one correspondence with their respective $j$ and $j'$ via
\eqref{indexAndResidual}). It is apparent from both the figure and the table
that the experimental growth rate in each case does indeed approach the rate
predicted by \autoref{cor:L2-approx} and \autoref{thm:lossyProjection} thus
corroborating the sharpness result \autoref{thm:optimality} proved for $r \in
\{0,1\}$ and supporting the conjecture \eqref{conjecturedRate} posed for $r
\geq 2$.

\begin{figure}
\includegraphics[width=\linewidth]{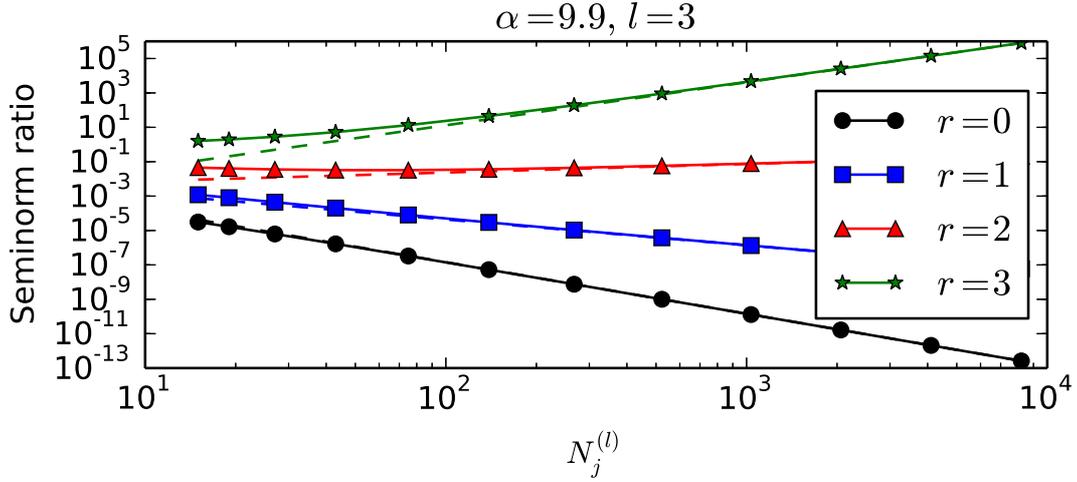}
\caption{Logarithmic plot of truncation degrees $N\psp{l}_j$ versus computed
seminorm ratios $\rat\psp{\alpha,l}_{r,j}$, $r \in \{0, \dotsc, l\}$, for one
instance of $\alpha$ and $l$. For comparison purposes plots of constant
positive scalar multiples of the function $N \mapsto N^{e(l,r)}$ (cf.\
\eqref{growthRate}) are also shown.}
\label{fig:cst}
\end{figure}

\begin{table}
\begin{tabular}{rl.l.l.l.}
$N^{(l)}_j$ & \ratio{\alpha}{l}{0}{j} & \egr{\alpha}{l}{0}{j} & \ratio{\alpha}{l}{1}{j} & \egr{\alpha}{l}{1}{j} & \ratio{\alpha}{l}{2}{j} & \egr{\alpha}{l}{2}{j} & \ratio{\alpha}{l}{3}{j} & \egr{\alpha}{l}{3}{j}\\\hline
15 & 3.11e-05 &   \nan & 1.20e-03 &   \nan & 4.55e-02 &   \nan & 1.61e+00 &   \nan\\
19 & 1.66e-05 & -2.665 & 8.06e-04 & -1.687 & 4.06e-02 & -0.480 & 1.96e+00 & 0.820\\
27 & 6.29e-06 & -2.754 & 4.44e-04 & -1.698 & 3.57e-02 & -0.369 & 2.83e+00 & 1.041\\
43 & 1.67e-06 & -2.847 & 2.02e-04 & -1.692 & 3.25e-02 & -0.205 & 5.28e+00 & 1.342\\
75 & 3.29e-07 & -2.921 & 8.02e-05 & -1.661 & 3.23e-02 & -0.008 & 1.34e+01 & 1.677\\
139 & 5.29e-08 & -2.965 & 2.96e-05 & -1.615 & 3.60e-02 & 0.174 & 4.54e+01 & 1.977\\
267 & 7.53e-09 & -2.986 & 1.06e-05 & -1.571 & 4.41e-02 & 0.311 & 1.91e+02 & 2.197\\
523 & 1.01e-09 & -2.994 & 3.77e-06 & -1.540 & 5.75e-02 & 0.397 & 9.17e+02 & 2.336\\
1035 & 1.30e-10 & -2.997 & 1.33e-06 & -1.522 & 7.80e-02 & 0.446 & 4.76e+03 & 2.414\\
2059 & 1.65e-11 & -2.999 & 4.72e-07 & -1.511 & 1.08e-01 & 0.472 & 2.58e+04 & 2.456\\
4107 & 2.08e-12 & -2.999 & 1.67e-07 & -1.506 & 1.51e-01 & 0.486 & 1.43e+05 & 2.478\\
8203 & 2.61e-13 & -3.000 & 5.90e-08 & -1.503 & 2.12e-01 & 0.493 & 7.99e+05 & 2.489\\
\end{tabular}
\caption{Truncation degrees, computed seminorm ratios for $r \in \{0, \dotsc,
l\}$ and experimental growth rates in the $\alpha = 9.9$ and $l = 3$ case.}
\label{tab:cst}
\end{table}

\appendix

\section{A density result}\label{sec:density}

In this appendix we will prove that, for all $d \in \natural$,
$\CC^\infty(\overline{B^d})$ is dense in the Sobolev-type space
$\WZ_\alpha(B^d)$ defined in \eqref{definition-of-WZ}. We will find it
convenient to define the index set
\begin{equation*}
\mathcal{I} := \left\{ (i,j) \mid i, j \in \{1, \dotsc, d\},\ i < j \right\}.
\end{equation*}
and denote
\begin{equation*}
D_{i,j} = x_j \, \partial_i - x_i \, \partial_j.
\end{equation*}
Given $\lambda > 0$ let us define the dilation operator $\delta_\lambda$ which
maps any function $f \colon \Real^d \to \Real$ to the function $\delta_\lambda
f \colon \Real^d \to \Real$ defined by
\begin{equation*}
(\forall\,x\in\Real^d) \quad \delta_\lambda f(x) := f(\lambda x).
\end{equation*}
Because of the change of variable properties of the Lebesgue integral with
respect to linear transformations (see, for example, \cite[Theorem~3.6.1 and
Corollary~3.6.4]{Bogachev}) $\delta_\lambda$ maps $\LL^p(\Real^d)$ into itself
and $\norm{\delta_\lambda f}_{\LL^p(\Real^d)} = \lambda^{-d/p}
\norm{f}_{\LL^p(\Real^d)}$.  Another consequence is that for any $f \in
\LL^1_{\mathrm{loc}}(\Real^d)$ and any multi-index $\beta \in \natural_0^d$,
$\partial_\beta (\delta_\lambda f) = \lambda^{\abs{\beta}} \delta_\lambda
(\partial_\beta f)$; hence,
\begin{equation}\label{dltDiffRules}
(\forall\,i\in\{1,\dotsc,d\}) \quad \partial_i(\delta_\lambda f) = \lambda \, \delta_\lambda(\partial_i f), \qquad
(\forall\,(i,j)\in\mathcal{I}) \quad D_{i,j}(\delta_\lambda f) = \delta_\lambda(D_{i,j} f).
\end{equation}

\begin{proposition}\label{pro:dilationContinuous} Let $p \in [1,\infty)$. For
all $f \in \LL^p(\Real^d)$, $\lim_{\lambda \to 1} \norm{f - \delta_\lambda
f}_{\LL^p(\Real^d)} = 0$.
\end{proposition}
\begin{proof} Let us assume for now that $f \in
\CC^\infty_0(\Real^d)$. Then there exists $R > 0$ such that
$\operatorname{supp}(f) \subset B(0, R)$. Also, restricted to $B(0, 2R)$, $f$
is uniformly continuous; that is,
\begin{equation*}
(\forall\,\epsilon > 0)\ (\exists\,\delta(\epsilon) > 0)\ (\forall\,x,y\in B(0,2R)) \quad \norm{x-y} < \delta(\epsilon) \implies \abs{f(x) - f(y)} < \epsilon.
\end{equation*}
Given any $\epsilon > 0$, let $\delta^*(\epsilon) := \min(\delta(\epsilon) / R,
1)$.  Then, for all $x \in B(0,R)$,
\begin{equation*}
\abs{\lambda - 1} < \delta^*(\epsilon)
\implies \norm{\lambda x - x} = \abs{\lambda - 1} \norm{x} < \delta(\epsilon)
\implies \abs{f(\lambda x) - f(x)} < \epsilon.
\end{equation*}
In other words, $\delta_\lambda f$ tends to $f$ uniformly in $B(0,R)$ as
$\lambda$ tends to $1$, whence $\lim_{\lambda \to 1} \int_{B(0,R)} \abs{f -
\delta_\lambda f}^p = 0$. On the other hand,
\begin{equation*}
\int_{\Real^d \setminus B(0,R)} \abs{f - \delta_\lambda f}^p
= \int_{B(0,R/\lambda) \setminus B(0,R)} \abs{\delta_\lambda f}^p
\leq \abs{B(0,R/\lambda) \setminus B(0,R)} \norm{f}_{\LL^\infty(\Real^d)}^p
\end{equation*}
which also tends to $0$ as $\lambda \to 1$, so $\delta_\lambda f$ tends to $f$
in $\LL^p(\Real^d)$.

Let now $f$ be an arbitrary member of $\LL^p(\Real^d)$ and let $\epsilon > 0$.
As $\CC^\infty_0(\Real^d)$ is dense in $\LL^p(\Real^d)$ (cf.\
\cite[Corollary~4.2.2]{Bogachev}) there exists $g \in \CC^\infty_0(\Real^d)$
such that $\norm{f - g}_{\LL^p(\Real^d)} < \epsilon/6$. From the argument above
we also know that there exists some $\delta^*$ such that $\norm{g -
\delta_\lambda g}_{\LL^p(\Real^d)} < \epsilon/2$ if $\abs{\lambda-1} <
\delta^*$. Now, if $\abs{\lambda - 1} < 1 - 2^{-p/d}$, $1 + \lambda^{-d/p} <
3$. Thus, if $\abs{\lambda - 1} < \min(\delta^*, 1-2^{-p/d})$,
\begin{multline*}
\norm{f - \delta_\lambda f}_{\LL^p(\Real^d)} + \norm{\delta_\lambda (f - g)}_{\LL^p(\Real^d)} + \norm{\delta_\lambda g - g}_{\LL^p(\Real^d)}\\
= (1+\lambda^{-d/p}) \norm{f-g}_{\LL^p(\Real^d)} + \norm{g - \delta_\lambda g}_{\LL^p(\Real^d)} < 3 \epsilon/6+\epsilon/2 = \epsilon.
\end{multline*}
\end{proof}

\begin{lemma}\label{lem:DSL} Let $\alpha \geq 0$ and $f \in \WZ_\alpha(B^d)$.
Then, $\delta_\lambda f \xrightarrow{\lambda \to 1^{-}} f$ in $\WZ_\alpha$.
\end{lemma}
\begin{proof} This proof is based on the proof of
\cite[Theorem~7.2]{Kufner:1985}. From the structure of the norm of
$\WZ_\alpha(B^d)$, the differentiation rules \eqref{dltDiffRules} and the
obvious fact that $\lim_{\lambda \to 1^{-}} \lambda = 1$ our desired result
would follow from
\begin{subequations}\label{shifts}
\begin{gather}
\label{dilation-id}%
\lim_{\lambda \to 1^{-1}} \norm{f - \delta_\lambda f}_{\LL^2_{\rho^\alpha}(B^d)} = 0,\\
\label{dilation-cart-diffs}%
(\forall\,i\in\{1,\dotsc,d\}) \quad \lim_{\lambda \to 1^{-}} \norm{\partial_i f - \delta_\lambda (\partial_i f)}_{\LL^2_{\rho^{\alpha+1}}(B^d)} = 0
\intertext{and}
\label{dilation-ang-diffs}%
(\forall\,(i,j) \in \mathcal{I}) \quad \lim_{\lambda \to 1^{-}} \norm{D_{i,j} f - \delta_\lambda (D_{i,j} f)}_{\LL^2_{\rho^\alpha}(B^d)} = 0,
\end{gather}
\end{subequations}
all of which we prove in the sequel. Let $(g, w)$ be any of $(f, \rho^\alpha)$,
$(\partial_i f, \rho^{\alpha+1})$ (for $i \in \{1, \dotsc, d\})$ or $(D_{i,j}
f, \rho^\alpha)$ (for $(i,j) \in \mathcal{I})$ and let $\epsilon > 0$. Upon
introducing
\begin{gather*}
J(\lambda) := \int_{B^d} \abs{g(x) - g(\lambda x)}^2 w(x) \dd x,\\
J_1(\lambda) := \int_{B^d} \abs{g(x) w(x)^{1/2} - g(\lambda x) w(\lambda x)^{1/2}}^2 \dd x
\intertext{and}
J_2(\lambda) := \int_{B^d} \abs{g(\lambda x)}^2 \abs{w(\lambda x)^{1/2} - w(x)^{1/2} }^2 \dd x
= \int_{B^d} \abs{g(\lambda x)}^2 w(\lambda x) \abs{\frac{w(x)^{1/2}}{w(\lambda x)^{1/2}} - 1}^2 \dd x
\end{gather*}
we find that $J(\lambda) \leq 2 \left[ J_1(\lambda) + J_2(\lambda) \right]$.
Now, applying \autoref{pro:dilationContinuous} to the extension by zero of
$w^{1/2} g \in \LL^2(B^d)$ we find that there exists $\lambda_1 \in (0,1)$ such
that $\lambda_1 < \lambda < 1$ implies $J_1(\lambda) < \epsilon^2/4$. From now
on we assume that $\lambda \in (\lambda_1,1)$. Now, given $\zeta \in (0,1)$,
the integral defining $J_2(\lambda)$ is the sum of the integral over $B(0,1)
\setminus B(0,\zeta)$ and the integral over $B(0,\zeta)$; we denote the former
by $J_{2,1}(\lambda,\zeta)$ and the latter by $J_{2,2}(\lambda,\zeta)$. As
$\lambda \in (0,1)$ and $w$ is monotonically non-increasing with respect to the
modulus of its argument,
\begin{equation*}
\frac{w(x)}{w(\lambda x)} \leq 1
\quad\text{and}\quad
1 - \left[\frac{w(x)}{w(\lambda x)}\right]^{1/2} \leq 1.
\end{equation*}
Thus,
\begin{equation*}
J_{2,1}(\lambda,\zeta)
\leq \int_{B(0,1) \setminus B(0,\zeta)} \abs{g(\lambda x)}^2 w(\lambda x) \dd x
= \lambda^{-d} \int_{B(0,\lambda) \setminus B(0,\lambda\zeta)} \abs{g(x)}^2 w(x) \dd x.
\end{equation*}
If $\zeta$ is close enough to $1$ the measure of the region $B(0,\lambda)
\setminus B(0,\lambda \zeta)$ can be made arbitrarily small. This, the fact
that $w^{1/2} g \in \LL^2(B^d)$, the absolute continuity of the Lebesgue
integral (cf.\ \cite[Theorem~2.5.7]{Bogachev}) and $0 < \lambda_1 < \lambda <
1$ imply that there exists $\zeta \in (0,1)$ such that $J_{2,1}(\lambda,\zeta)
< \epsilon^2/8$. Let us fix such $\zeta$. Then,
\begin{equation*}
\begin{split}
J_{2,2}(\lambda,\zeta)
& \leq \sup_{x \in B(0,\zeta)} \abs{\frac{w(x)^{1/2}}{w(\lambda x)^{1/2}} - 1}^2 \int_{B(0,\zeta)} \abs{g(\lambda x)}^2 w(\lambda x) \dd x\\
& = \sup_{x \in B(0,\zeta)} \abs{\frac{w(x)^{1/2}}{w(\lambda x)^{1/2}} - 1}^2 \lambda^{-d} \int_{B(0,\lambda \zeta)} \abs{g(x)}^2 w(x) \dd x.
\end{split}
\end{equation*}
As the integral in the last expression is bounded by
$\norm{g}_{\LL^2_w(B^d)}$, $\lambda^{-d} < \lambda_1^{-d}$ and
\begin{equation*}
\lim_{\lambda \to 1^-} \left( 1 - \left[\frac{w(x)}{w(\lambda x)}\right]^{1/2} \right) = 0
\end{equation*}
uniformly in $B(0,\zeta)$, we conclude that there exists $\lambda_2 \in
(\lambda_1, 1)$ such that $J_{2,2}(\lambda,\zeta) < \epsilon^2/8$ if $\lambda_2
< \lambda < 1$. Combining this with the other obtained bounds we have that
$J(\lambda) = \norm{g - \delta_\lambda g}_{\LL^2_w(B^d)}^2 < \epsilon^2$ if
$\lambda_2 < \lambda < 1$ and thus all the limits appearing in \eqref{shifts}
have been proved.
\end{proof}

\begin{corollary}\label{cor:WZ-density-nonNeg}
If $\alpha \geq 0$, then $\CC^\infty(\overline{B^d})$ is
dense in $\WZ_\alpha(B^d)$.
\end{corollary}
\begin{proof} Let $f \in \WZ_\alpha(B^d)$ and let $\epsilon > 0$. From
\autoref{lem:DSL} we know there exists some $\lambda \in (0,1)$ such that
$\norm{f - \delta_\lambda f}_{\WZ_\alpha(B^d)} < \epsilon$. Because of the
scaling of the argument, $\partial_\lambda f$ belongs to the standard Sobolev
space $\HH^1(B^d)$, whence there exists $\varphi \in
\CC^\infty(\overline{B^d})$ such that $\norm{\delta_\lambda f -
\varphi}_{\HH^1(B^d)} \leq \epsilon$ (cf.\
\cite[Theorem~3.22]{AF:2003}). As (because $\alpha \geq 0$) all the weight functions involved are bounded by $1$ and the functions $x \mapsto x_j$ appearing in the definition of the differential operators $D_{i,j}$ are bounded, $\WZ_\alpha(B^d) \subset \HH^1(B^d)$ with continuous embedding. Thus, $\norm{f - \varphi}_{\WZ_\alpha(B^d)} \leq (1+C) \epsilon$, where $C$ is the constant of the aforementioned embedding.
\end{proof}

\bibliographystyle{amsplain}
\bibliography{za-refs}

\end{document}